\documentclass{article}

\usepackage{amsmath, amssymb, amsthm}
\usepackage{enumerate}
\usepackage{graphicx}

\newcommand{\x}{\times}
\newcommand{\bit}{\begin{itemize}}
\newcommand{\eit}{\end{itemize}}
\newcommand{\ben}{\begin{enumerate}}
\newcommand{\een}{\end{enumerate}}
\newcommand{\beq}{\begin{equation}}
\newcommand{\eeq}{\end{equation}}
\newcommand{\bea}{\begin{eqnarray*}}
\newcommand{\eea}{\end{eqnarray*}}
\newcommand{\bpf}{\begin{proof}}
\newcommand{\epf}{\end{proof}\ms}
\newcommand{\bmt}{\begin{bmatrix}}
\newcommand{\emt}{\end{bmatrix}}
\newcommand{\ms}{\medskip}

\newcommand{\OL}{\overline}
\newcommand{\noi}{\noindent}
\newcommand{\wh}{\widehat}
\newcommand{\eps}{\varepsilon}
\usepackage{bm}
\usepackage{color}
\definecolor{red}{rgb}{1,0,0}

\definecolor{green}{rgb}{0,.7,0}

\definecolor{blue}{rgb}{0,0,1}

\newcommand{\ba}{\mathbf a}
\newcommand{\bb}{\mathbf b}

\usepackage{tikz}

\newtheorem{theorem}{Theorem}[section]
\newtheorem{lemma}[theorem]{Lemma} 
\newtheorem{proposition}[theorem]{Proposition}
\newtheorem{corollary}[theorem]{Corollary}
\newtheorem{observation}[theorem]{Observation}
\newtheorem{question}[theorem]{Question}
\newtheorem{conj}[theorem]{Conjecture}

\theoremstyle{definition}
\newtheorem{definition}[theorem]{Definition}
\newtheorem{remark}[theorem]{Remark}
\newtheorem{example}[theorem]{Example}

\newenvironment{thm}{\begin{theorem}}{\end{theorem}}
\newenvironment{lem}{\begin{lemma}}{\end{lemma}}
\newenvironment{prop}{\begin{proposition}}{\end{proposition}}
\newenvironment{cor}{\begin{corollary}}{\end{corollary}}
\newenvironment{obs}{\begin{observation}}{\end{observation}}

\newenvironment{defn}{\begin{definition}}{\end{definition}}
\newenvironment{rem}{\begin{remark}}{\end{remark}}
\newenvironment{ex}{\begin{example}}{\end{example}}

\numberwithin{figure}{section}  
\numberwithin{table}{section}   

\newcommand{\sP}{\mathcal{P}}
\newcommand{\sQ}{\mathcal{Q}}
\newcommand{\sA}{\mathcal{A}}
\newcommand{\sB}{\mathcal{B}}
\newcommand{\sC}{\mathcal{C}}

\newcommand{\plus}{\raisebox{.4\height}{\scalebox{0.6}{$\bm{+}$}}}
\newcommand{\minus}{\raisebox{.2\height}{\scalebox{0.6}{$\bm{-}$}}}
\newcommand{\st}{\raisebox{.1\height}{\scalebox{0.7}{$\bm{*}$}}}

\newcommand{\Ap}{{\tt A}^{\plus}}
\newcommand{\Am}{{\tt A}^{\minus}}
\newcommand{\Ab}{{\tt A}^{\st}}
\newcommand{\Aa}{{\tt A}}
\newcommand{\Sa}{{\tt S}}
\newcommand{\Sp}{{\tt S}^{\plus}}
\newcommand{\Sm}{{\tt S}^{\minus}}
\newcommand{\Sb}{{\tt S}^{\st}}
\newcommand{\NN}{{\tt N}}
\newcommand{\XX}{{\tt X}}
\newcommand{\SEPR}{\operatorname{SEPR}}
\newcommand{\diag}{\operatorname{diag}}
\newcommand{\trans}{^\top}
\newcommand{\dunion}{\mathbin{\dot{\cup}}}
\newcommand{\negg}{\operatorname{neg}}
\newcommand{\Pl}{\stackrel{\circ}{P}}
\newcommand{\Sl}{\stackrel{\circ}{S}}
\newcommand{\Kl}{\stackrel{\circ}{K}}
\newcommand{\oo}{\circ\circ}
\newcommand{\Pll}{\stackrel{\oo}{P}}

\DeclareMathOperator{\sgn}{sgn}
\newcommand{\D}{\Gamma}

\DeclareMathOperator{\sepr}{sepr}
\DeclareMathOperator{\match}{match}
\def\mtx#1{\begin{bmatrix} #1 \end{bmatrix}}

\title{The sepr-sets of sign patterns}

\author{Leslie
Hogben\thanks{Department of Mathematics, Iowa State University,
Ames, IA 50011, USA (hogben@iastate.edu) and American Institute of Mathematics, 600 E. Brokaw Road, San Jose, CA 95112, USA
(hogben@aimath.org).}
\and Jephian C.-H. Lin\thanks{Department of
Mathematics and Statistics, University of Victoria, Victoria, BC,
V8W 2Y2, Canada (jephianlin@gmail.com).}
\and D.~D.~Olesky\thanks{Department of Computer Science, University of Victoria, BC,
V8W 2Y2, Canada (dolesky@cs.uvic.ca).}
\and P. van den Driessche\thanks{Department of
Mathematics and Statistics, University of Victoria, Victoria, BC,
V8W 2Y2, Canada (pvdd@math.uvic.ca).  Research supported in part by
an NSERC  Discovery grant.}}

\begin{document}
\maketitle

\begin{abstract}
Given a real symmetric $n\times n$ matrix, the sepr-sequence $t_1\cdots t_n$ records information about the existence of principal minors of each order that are positive, negative, or zero. This paper extends the notion of the sepr-sequence to matrices whose entries are of prescribed signs, that is, to sign patterns. A sufficient condition is given for a sign pattern to have a unique sepr-sequence, and it is conjectured to be necessary. The sepr-sequences of sign semi-stable patterns are shown to be well-structured; in some special circumstances, the sepr-sequence is enough to guarantee the sign pattern being sign semi-stable. In alignment with previous work on symmetric matrices, the sepr-sequences for sign patterns realized by  symmetric nonnegative matrices of orders two and three are characterized.  

\medskip
\noindent{\bf Keywords:} Signed enhanced principal rank characteristic sequence, sign pattern, sign semi-stable pattern, principal minor, digraph

\medskip
\noindent{\bf AMS subject classification:} 15B35, 15A15, 15B48, 05C50
\end{abstract}


\section{Introduction}

There are numerous problems where
the principal minors of a matrix, or their signs, provide important information; examples include determining whether a matrix is a $P$-matrix (all the principal minors are positive), and in the solvability of the inverse multiplicative eigenvalue problem \cite{F75}
(see \cite{GT06} for further examples).
 The {principal minor assignment problem} \cite{HS02}   asks:  Given the values for each principal minor, is there a symmetric matrix whose principal minors agree with the corresponding given values?  Oeding  \cite{O11} solved the principal minor assignment problem  for complex symmetric matrices, but the original question for real symmetric matrices remains  open.

To provide partial answers to problems requiring information about nonsingularity of principal submatrices, the {principal rank characteristic sequence (pr-sequence)} of a symmetric matrix was introduced in \cite{BDOvdD12}; it describes the existence or nonexistence of nonzero minors of the matrix.  Subsequently more refined forms were introduced, including  the {enhanced principal rank characteristic sequence  (epr-sequence)} \cite{BCFHHvdDY16} and the {signed enhanced principal rank characteristic sequence (sepr-sequence)} introduced by Mart\'inez-Rivera in \cite{Xavier17}.   As noted in \cite{GT06}, qualitative information, such as the signs of the principal minors, is sufficient for some applications 
 (and sometimes that is the only information available), making the sepr-sequence particularly valuable. 

All of the definitions of the pr-, epr-, and sepr-sequences have been stated only for symmetric (or complex Hermitian) matrices.  However, the same definitions apply naturally to real square matrices that are not necessarily symmetric. In the next definition we formally state the definition of  the sepr-sequence for this larger class of matrices. 

\begin{defn}
Let $B$ be an $n\times n$ real matrix.  The \emph{signed enhanced principal rank characteristic sequence} (sepr-sequence) of $B$ is a sequence $\sepr(B)=t_1t_2\cdots t_n$, where $t_k$ is one of $\Ab$, $\Ap$, $\Am$, $\NN$, $\Sb$, $\Sp$, or $\Sm$  based on the following criteria: For each  $k=1,\ldots, n$, $t_k=\Ab$ if $B$ has both a positive
and a negative order-$k$ principal minor, and each order-$k$ principal minor is nonzero; $t_k = \Ap$ (respectively, $t_k = \Am$) if each order-$k$ principal minor is positive (respectively, negative); $t_k = \NN$ if each order-$k$ principal minor is zero; $t_k = \Sb$ if $B$ has each a positive, a negative, and a zero order-$k$ principal minor; $t_k = \Sp$ (respectively, $t_k = \Sm$) if $B$ has both a zero and a nonzero order-$k$ principal minor, and each nonzero order-$k$ principal minor is positive (respectively, negative).
The $k$-th term of $\sepr(B)$ is also  denoted by $t_k(B)$.
\end{defn}

 Given that  the sepr-sequence of a matrix summarizes information about signs of principal minors, it seems natural to study what happens when we know the signs of the entries of the matrix but not their actual values.  
 A \emph{sign pattern} is a matrix whose entries are the \emph{signs} $\{+,-,0\}$.  The \emph{qualitative class} $Q(\sP)$ of a sign pattern $\sP=\begin{bmatrix}p_{ij}\end{bmatrix}$ consists of all  real matrices $B=\begin{bmatrix}b_{ij}\end{bmatrix}$ with the same dimensions as $\sP$ such that the sign of $b_{ij}$ is $p_{ij}$ for each $i$ and $j$.  For an $n\x n$ sign pattern $\sP$, $n$ is called the {\em order} of $\sP$ and is denoted by  $n(\sP)$.

The study of sign patterns arises from the study of  dynamical systems that appear in economics and biology, in the sense that the linearization of an equilibrium is often described by a sign pattern but the precise value of each entry is often unknown (see, for example, Sections 1.1 and 10.1 in \cite{BS95}).  Many sign patterns that arise in such applications are not symmetric.   Efforts have been made to characterize sign patterns $\sP$ that guarantee certain spectral or determinantal properties, such as sign stable patterns (every matrix in $Q(\sP)$ has all  eigenvalues in the left-half of the complex plane) and sign nonsingular patterns (every matrix in the qualitative class is nonsingular); see, e.g., \cite{BS95} and the references therein.  The principal minors of a matrix are a generalization of the determinant and also capture the spectrum through the characteristic polynomial.  We  study the possible signs of principal minors of  sign patterns with the help of  sepr-sequences.  

\begin{defn}
Let $\sP$ be a square sign pattern.  The \emph{sepr-set} of $\sP$, denoted by $\SEPR(\sP)$, is the set of sepr-sequences of matrices in $Q(\sP)$.  A sign pattern $\sP$ has a {\em unique sepr-sequence} if $\SEPR(\sP)$ has exactly one  element; in this case we denote the unique element of $\SEPR(\sP)$ by $\sepr(\sP)$.
\end{defn}

  In Section \ref{s:unique} we establish sufficient conditions for a sign pattern to have a unique sepr-sequence and conjecture these conditions are also necessary; necessity is established in some special cases. We also prove certain conditions are necessary  for an sepr-sequence  to be the unique sepr-sequence of a sign pattern.  We establish properties of sepr-sequences of a sign semi-stable pattern and determine exactly the sepr-sequences attainable by irreducible sign semi-stable patterns with all diagonal entries zero. Section \ref{s:nnsym} contains results about sepr-sequences of symmetric nonnegative matrices and symmetric nonnegative sign patterns, including a determination of all sepr-sequences attainable by symmetric nonnegative matrices of order at most three, and results about symmetric nonnegative sign patterns with unique sepr-sequences.
The remainder of this introduction contains additional definitions and notation, and Section \ref{s:basic} describes basic properties of sepr-sequences of nonsymmetric matrices and of sign patterns, including a description of how to compute sepr-sequences of reducible matrices and sign patterns.  Section \ref{s:conclude} has concluding remarks, including directions for future research.

Multiplication and addition of real numbers naturally induce a multiplication and  addition of signs.  That is, we make the obvious conventions that $+\cdot+=+=-\cdot-$, $+\cdot-=-=-\cdot+$, 0 times anything is 0, and   $s+s=s=s+0$ for $s\in\{+,-,0\}$.  However, any formula  with addition of $+$ and $-$ involved is called \emph{ambiguous}.  For an $n\times n$ sign pattern $\sP=\begin{bmatrix} p_{ij}\end{bmatrix}$, the determinant is defined by the \emph{standard expression}
\[\det\sP=\sum_{\sigma}\operatorname{sgn}(\sigma)\prod_{i=1}^n p_{i\sigma(i)},\]
where the sum is one of $+, -, 0$, or ambiguous,  $\sigma$ runs through all permutations on $\{1,\ldots,n\}$, and $\operatorname{sgn}(\sigma)$ is the sign of the permutation.  For example, $\det\begin{bmatrix}+&+\\+&-\end{bmatrix}=-$ and $\det\begin{bmatrix}+&+\\-&-\end{bmatrix}$ is ambiguous.  

{An $n\times n$ sign pattern $\sP$ has a \emph{signed determinant} if the determinants of all matrices in $Q(\sP)$  have the same sign.  It is known (see, e.g., \cite[Lemma 1.2.4]{BS95})
that a sign pattern has a signed determinant if and only if $\det\sP$ is one of $+$, $-$, or $0$ (i.e., $\det\sP$ is  not ambiguous).
An \emph{ambiguous sign pattern} is a sign pattern whose determinant is ambiguous.  The notation $\det \sP\neq 0$ means there is at least one nonzero term in the standard expression.  In the proof of \cite[Lemma 1.2.4]{BS95}, it is shown that every ambiguous sign pattern allows a positive determinant and a negative determinant in its qualitative class, so it also allows a zero determinant by continuity.}

Let $\sP$ be an $n\x n$ sign pattern and define $[n]=\{1,\dots,n\}$.  If $\alpha,\beta\subseteq [n]$ are indices of rows and columns, respectively, then $\sP[\alpha,\beta]$ is the subpattern\footnote{The reader is warned that in the sign pattern literature  `subpattern of $\sP$' is sometimes used to mean a pattern obtained from $\sP$ by changing some nonzero entries to 0; we do not use `subpattern' in that way.}  of $\sP$ induced by rows in $\alpha$ and columns in $\beta$.  When $\alpha=\beta$, the subpattern $\sP[\alpha]=\sP[\alpha,\alpha]$ is called a \emph{principal subpattern}.  Also, $\sP(\alpha,\beta)$ is the subpattern of $\sP$ induced by rows outside $\alpha$ and columns outside $\beta$; the principal subpattern $\sP(\alpha, \alpha)$ is also denoted by $\sP(\alpha)$.   This notation is also applied to matrices. 

Digraphs and signed digraphs play a central role in the study of sign patterns.  A {\em digraph} is a pair of sets $\D=(V,E)$ where $E\subseteq V\x V$;  the elements of $v$ are {\em vertices} and the elements of $E$ are called {\em arcs} or {\em directed edges}.  Note that  multiple copies of an arc are not allowed, but $(u,v)$ and $(v,u)$ are considered as different and thus both are permitted.  A {\em loop} is an arc of the form $(v,v)$.
  A \emph{signed digraph} $\D=(V,E,\Sigma)$ is a digraph in which each arc is assigned a sign, i.e., $\Sigma:E\to \{+,-\}$.  A {\em $k$-cycle} is a digraph with vertex set $\{v_1,\dots,v_k\}$ and arc set $\{(v_1,v_2),\dots,(v_{k-1},v_k),(v_k,v_1)\}${, where $k\geq 1$}; $k$ is the {\em length} of a $k$-cycle.  A {\em composite cycle} is a union of one or more disjoint cycles. For a $k$-cycle $C$ in a signed digraph, the {\em cycle product} $\prod(C)$ of $C$ is the product of the signs on the arcs of $C$, and the {\em signed cycle product} of  $C$ is $(-)^{k+1}\prod(C)$.  The terms `cycle product' and `signed cycle product' are also applied to composite cycles by multiplying the parts corresponding to the cycles in the composite cycle.  These terms are also applied to sign patterns and matrices, but only cycles or composite cycles in the digraph are considered (that is, we only consider nonzero cycle products in sign patterns and matrices). 

A (simple undirected) graph $G=(V,E)$  has a set of vertices  and an edge set consisting of two-element subsets of vertices. 
A {\em matching}   is a set $M\subseteq E$ of disjoint edges of $G$; $M$ is a matching of $U \subseteq V$ if 
every vertex in $U$ appears in an edge of $M$.  A matching of $V$ is a {\em perfect matching} of $G$.  The {\em matching number} of a  graph $G$ is the maximum number of edges  in a matching of $G$  and is denoted by $\match(G)$.  The {\em underlying graph} of a digraph $\D=(V,E)$ is the graph $G=(V,\wh E)$ where for $u\ne v$ the edge $\{u,v\}$ is in  $\wh E$ if and only if at least one of the arcs $(u,v)$, $(v,u)$ is in $E$ (and loops in $\Gamma$ are ignored).  

For a digraph $\Gamma =(V,E)$, a {\em subdigraph} is a digraph $\hat\Gamma =(\hat V,\hat E)$ such that $\hat V\subseteq V$ and $\hat E\subseteq E$; $\hat \Gamma$ is {\em induced} if $\hat E =E\cap (\hat V\x \hat V)$, in which case we write  $\hat \Gamma=\Gamma[\hat V]$.  Subgraphs and  induced subgraphs are defined analogously.

A digraph $\D$ is {\em strongly connected} if {for any two distinct vertices $u$ and $v$ there is a walk 
\[(u,w_1),(w_1,w_2),\ldots,(w_k,v)\in E(\D)\]
connecting $u$ and $v$}.  A {\em strong component} of a digraph is a maximal strongly connected induced subdigraph.
A digraph $\D=(V,E)$ is {\em doubly directed} if for every $u,v\in V$, $(u,v)\in E \iff (v,u)\in E$. 
A signed doubly directed digraph $\D=(V,E)$ is {\em skew-symmetric} if for every $(u,v)\in E$, $\Sigma(u,v)=-\Sigma(v,u)$; such a digraph does not have a loop.  Similarly, a signed doubly directed digraph $\D=(V,E)$ is \emph{symmetric} if for every $(u,v)\in E$, $\Sigma(u,v)=\Sigma(v,u)$, with $u$ and $v$ possibly the same. 
A {\em strong ditree} {(or \emph{diforest})} is  a doubly directed digraph  whose underlying graph is a tree (or forest).    The {\em matching number} of a strong ditree or diforest $\Gamma$, denoted by $\match(\Gamma)$, is the matching number of its underlying graph.

The {\em signed digraph $\D(\sP)$} of an $n\times n$ sign pattern $\sP=\begin{bmatrix}p_{ij}\end{bmatrix}$ has vertex set $\{1,\ldots,n\}$, arc set $\{(i,j):p_{ij}\neq 0\}$, and  $\Sigma(i,j)=p_{ij}$; the same terminology is applied to matrices.
   The \emph{simplified} pattern $\sQ$ of $\sP$ is obtained from $\sP$ by setting to zero every $i,j$-entry such that $(i,j)$ is not part of a cycle in $\Gamma(\sP)$ (i.e.,  $(i,j)$ is an arc with endpoints in two different strong components in $\Gamma(\sP)$).  

\begin{rem}\label{rem:simp} Let $\sQ$ be the simplified pattern   of a sign  pattern $\sP$.
Since the removed arcs are not in any cycle, $\SEPR(\sQ)=\SEPR(\sP)$. 
\end{rem}


\section{Basic properties of sepr-sequences of nonsymmetric matrices and sign patterns}\label{s:basic}
 

In this section we present some basic properties of sepr-sequences of not necessarily symmetric matrices and  sign patterns.  
In an sepr-sequence,   $\OL{\XX_i\cdots \XX_j}$ indicates that the complete 
sequence $\XX_i\cdots \XX_j$ may be repeated as many times as desired (or may be omitted entirely).

Many sepr-sequences that are forbidden for symmetric matrices are realized by nonsymmetric matrices. These include fundamental results that apply to pr- and epr-sequences as well.  Here we list two such examples.  The $\NN\NN$ Theorem  \cite[Theorem 2.3]{Xavier17}  says that all terms in an sepr-sequence after  two consecutive terms equal to $\NN$ must be $\NN$  for real symmetric and Hermitian matrices, whereas a matrix $B$ whose digraph is an $n$-cycle with $n\ge 3$ has $\sepr(B)=\NN\NN\OL{\NN}\Ap$ or $\sepr(B)=\NN\NN\OL{\NN}\Am$.  The $\NN\Sa\Aa$ Theorem  \cite[Corollary 1.3]{Xavier17}  says that the subsequence $\NN\Sa\Aa$ is prohibited in the epr-sequence of a real symmetric or Hermitian matrix (so $\NN\Sm\Ap$ etc. are prohibited in an sepr-sequence), whereas $\sepr(B)=\NN\Sm\Ap$ for 
$B=\mtx{
0 & 1 & 0 \\
 1 & 0 & 1 \\
 1 & 0 & 0  }$. 

Some properties established in \cite{Xavier17}  do  remain true for not necessarily symmetric matrices.

\begin{obs} The sepr-sequence of a square  real matrix must end in $\Ap, \Am$, or $\NN$.
\end{obs}

The proof of  \cite[Theorem 2.4]{Xavier17} uses Jacobi's determinantal identity and not symmetry, so it remains valid and establishes the next result.
\begin{thm}[Inverse Theorem] Suppose $B$ is a nonsingular real square matrix.
\ben[{\rm (i)}]
\item If $\sepr(B) = t_1t_2 \cdots t_{n-1}\Ap$, then $\sepr (B^{-1}) = t_{n-1}t_{n-2} \cdots t_1\Ap$.
\item If $\sepr(B) = t_1t_2 \cdots t_{n-1}\Am$, then $\sepr (B^{-1}) = \negg(t_{n-1}t_{n-2} \cdots t_1\Ap)$ where $\negg (t_1t_2 \cdots t_k)$
is the sequence resulting from replacing $+$ superscripts with $-$ superscripts in $t_1t_2 \cdots t_k$,
and vice versa.
\een
\end{thm}

A sign pattern $\sP$ has a {\em fixed $k$-th sepr term} if $t_k(B)=t_k(B')$ for all $B,B'\in Q(\sP)$; in this case  $t_k(\sP)$ denotes this common value $t_k(B)$ for  $B\in Q(\sP)$.

\begin{obs} Let $\sP$ be an $n\x n$ sign pattern.
\ben[{\rm 1.}]
\item $\sP$ has a fixed $1$st sepr term, and any of $\Ab$, $\Ap$, $\Am$, $\NN$, $\Sb$, $\Sp$, or $\Sm$ is possible for $t_1(\sP)$.
\item If $\sP$ has a fixed $n$-th sepr term, then $t_n(\sP)\in \{\Ap,\Am,\NN\}$.
\item If every order $k$ principal subpattern of $\sP$ has signed determinant, then $\sP$ has a fixed $k$-th sepr term.
\een
\end{obs}

The next result plays an important role in the study of sign patterns that have unique sepr-sequences, but applies more generally to fixed sepr terms.

\begin{prop}\label{prop:fix-k}
Let $\sP$ be an $n\times n$ sign pattern  such that for some $k\in[n]$ either every $k\times k$ principal subpattern has a signed determinant, or there are three $k\times k$ principal subpatterns that have signed determinants equal to $+$, $-$, and $0$, respectively.
Then $\sP$ has a fixed $k$-th sepr term.
\end{prop}
\bpf
  If every $k\times k$ principal subpattern has a signed determinant, then clearly the $k$-th term of the sepr-sequence is also determined independent of the choice of a  matrix realization.  If there are $k\times k$ principal subpatterns that have signed determinants $+$, $-$, and $0$, then $t_k(B)=\Sb$ regardless of the  other $k\times k$ principal subpatterns and of the choice of the matrix  $B\in Q(\sP)$.
\epf

\begin{rem}\label{rem:Sstar}  Suppose that  $\sP$ has a fixed $k$-th sepr term and an ambiguous $k\x k$ principal subpattern $\sP[\alpha]$.  This implies there exist $B_+, B_-, B_0\in Q(\sP[\alpha])$ such that $\sgn(\det B_+[\alpha])=+$, $\sgn(\det B_-[\alpha])=-$, and $\det B_0[\alpha]=0$.  Then the fixed $k$-th sepr term implies $t_k(\sP)=\Sb$.
\end{rem}

\begin{prop}\label{prop:t1}  Let $\sP$ be an $n\x n$ sign pattern.  
\ben[{\rm (i)}]
\item\label{c:Ai}  If $t_1(\sP)=\Ap$, then $ \Ap\OL{\Ap}\in \SEPR(\sP).$
\item\label{c:Aii} If $ t_1(\sP)=\Am$ and $n$ is odd, then $ \Am\OL{\Ap\Am}\in \SEPR(\sP)$. 
\item\label{c:Aiii} If $ t_1(\sP)=\Am$ and $n$ is even, then $ \Am\OL{\Ap\Am}\Ap\in \SEPR(\sP)$. 
\een
Furthermore, every sepr-sequence of the forms \eqref{c:Ai}, \eqref{c:Aii}, and \eqref{c:Aiii} is attainable as  the  sepr-sequence of a diagonal sign pattern.  
\end{prop}
\bpf The stated all-{\tt A} sepr-sequences can be realized by choosing matrices with all diagonal entries  $\pm 1$ and every off-diagonal entry  0 or $\pm\eps$ where $\eps=\frac 1{n!}$.
It is straightforward to verify that a diagonal sign pattern with all $+$ or all $-$ on the diagonal gives the desired sepr-sequence.
\epf


A square matrix or sign pattern $M$  is {\em reducible} if there exists a permutation matrix $P$ such that $P\trans MP=\mtx{A & C\\ O & B}$; otherwise, it is \emph{irreducible}.
It is well-known that a matrix or sign pattern is irreducible if 
and only if its digraph is strongly connected. %

To provide compact notation for determining the sepr-sequence of a reducible matrix from its irreducible parts, we list  definitions of  addition and  multiplication of the symbols $\NN , \Ap , \Am , \Ab , \Sp , \Sm , \Sb$ in the tables below, and then define a rule $\ba*\bb$ for combining sepr-sequences.
\[
\begin{array}{c|ccccccc}
 + & \NN & \Ap & \Am & \Ab & \Sp & \Sm & \Sb \\
\hline 
\NN & \NN & \Sp & \Sm & \Sb & \Sp & \Sm & \Sb \\
\Ap & \Sp & \Ap & \Ab & \Ab & \Sp & \Sb & \Sb \\
\Am & \Sm & \Ab & \Am & \Ab & \Sb & \Sm & \Sb \\
\Ab & \Sb & \Ab & \Ab & \Ab & \Sb & \Sb & \Sb \\
\Sp & \Sp & \Sp & \Sb & \Sb & \Sp & \Sb & \Sb \\
\Sm & \Sm & \Sb & \Sm & \Sb & \Sb & \Sm & \Sb \\
\Sb & \Sb & \Sb & \Sb & \Sb & \Sb & \Sb & \Sb \\
\end{array}
\qquad
\begin{array}{c|ccccccc}
 \cdot & \NN & \Ap & \Am & \Ab & \Sp & \Sm & \Sb \\
\hline 
\NN & \NN & \NN & \NN & \NN & \NN & \NN & \NN \\
\Ap & \NN & \Ap & \Am & \Ab & \Sp & \Sm & \Sb \\
\Am & \NN & \Am & \Ap & \Ab & \Sm & \Sp & \Sb \\
\Ab & \NN & \Ab & \Ab & \Ab & \Sb & \Sb & \Sb \\
\Sp & \NN & \Sp & \Sm & \Sb & \Sp & \Sm & \Sb \\
\Sm & \NN & \Sm & \Sp & \Sb & \Sm & \Sp & \Sb \\
\Sb & \NN & \Sb & \Sb & \Sb & \Sb & \Sb & \Sb \\
\end{array}
\]
Given two sepr-sequences $\ba=a_1a_2\cdots a_n$ and $\bb=b_1b_2\cdots b_m$, define
\[\label{eq:red} \ba*\bb=t_1\cdots t_{n+m} \mbox{ where } t_k=\sum_{\ell=0}^k a_\ell\cdot b_{k-\ell};\]
 by convention $a_0=b_0=\Ap$, and products that include $a_j$ with $j>n$, and $b_j$ with $j>m$ or $j<0$, are ignored.  Note that although the product of the symbols $\NN , \Ap$, etc.~is an actual product (reflecting a product of minors), the ``sum" here is not a sum at all, but  the symbol obtained when combining the existence of certain signs of minors. 
\begin{ex}\label{ex:reduce} Suppose  $\ba=\Sp\NN$ and  $\bb=\Ap\Sp\Am$.  Then  $\ba*\bb=\Sp\Sp\Sb\Sm\NN$ because
\[\begin{aligned}
t_1 &= \Ap\cdot\Ap+\Sp\cdot\Ap=\Ap+\Sp=\Sp, \\
t_2 &=\Ap\cdot\Sp+\Sp\cdot \Ap + \NN\cdot\Ap = \Sp+\Sp+\NN=\Sp, \\
t_3 &= \Ap\cdot\Am+\Sp\cdot\Sp+\NN\cdot\Ap=\Am+\Sp+\NN=\Sb, \\
t_4 &= \Sp\cdot\Am+\NN\cdot\Sp=\Sm+\NN=\Sm,\\
t_5 &= \NN\cdot\Am=\NN.\\
\end{aligned}\]
 \end{ex}

When the definitions of $+, \cdot$, and $*$ are understood, the next result is immediate. 

\begin{prop}\label{prop:reduce}
If $A$ and $B$ are square matrices and $C$ is a matrix of appropriate dimensions, then  
\[\sepr\left(\mtx{ A & C \\ O & B }\right)=\sepr(A)*\sepr(B).\]
If $\sA$ and $\sB$ are square sign patterns  and $\sC$ is a sign pattern of appropriate dimensions, then \[\SEPR\left(\mtx{ \sA & \sC \\ O & \sB }\right)=\{\ba*\bb : \ba\in \SEPR(\sA), \bb\in \SEPR(\sB)\}.\]
\end{prop}

As a result of Proposition \ref{prop:reduce}, for many purposes it is sufficient to determine the sepr-sequences of irreducible sign patterns.

\section{Sign patterns with unique sepr-sequences} \label{s:unique}
 
In this section we give sufficient conditions for a sign pattern to have a unique sepr-sequence, conjecture that these conditions are necessary,  and determine various sign patterns that have unique sepr-sequences.  We also determine necessary conditions for an sepr-sequence to be the  unique sepr-sequence of a sign pattern.  Finally, we study sepr-sequences of sign semi-stable matrices and other sign patterns with similar structural properties.  We begin with some simple results and examples.

 If $\sP$ and $\sQ$ each have a unique sepr-sequence, then $\sP\oplus \sQ$ also has a unique sepr-sequence by Proposition \ref{prop:reduce}.  The sign patterns in the next example show that $\sP\oplus\sQ$ can have a unique sepr-sequence even though neither of them does. 
 
 \begin{ex}\label{ex:red-unique} Let 
\[\sP=\begin{bmatrix}
+ & + & 0 \\
- & - & + \\
0 & + & 0 
\end{bmatrix}\text{ and }
\sQ=\begin{bmatrix}
- & + & - \\
- & + & + \\
- & - & 0 
\end{bmatrix}.\]
Then $\SEPR(\sP)=\{\Sb\Sm\Am, \Sb\Sb\Am\}$ and $\SEPR(\sQ)=\{\Sb\Ab\Am,\Sb\Sb\Am\}$ are not unique.  However, $\sepr(\sP\oplus\sQ)=\{\Sb\Sb\Sb\Sb\Sb\Ap\}$ is unique.
\end{ex}

It is immediate that if every principal subpattern of a sign pattern $\sP$ has a signed determinant, then $\sP$ has a unique sepr-sequence.  Some important classes of sign patterns, such as sign semi-stable patterns, have this property (see Section \ref{s-sss}).  But there are sign patterns with unique sepr-sequences for which some principal subpatterns do not have a signed determinant; see  {Example~\ref{ex:red-unique}, or  Example~\ref{ex:NSDirreducible} for an irreducible example.}

\begin{ex}
\label{ex:NSDirreducible}
Let
\[\sP=\begin{bmatrix}
+ & + & 0 & 0 \\
0 & - & + & 0 \\
+ & 0 & + & + \\
0 & 0 & - & 0
\end{bmatrix}.\]
The principal subpattern $\sP[\{1,2,3\}]$ does not have a signed determinant.  However, $\sP$ has the unique sepr-sequence $\Sb\Sb\Sb\Am$.  Note that if we change the $3,4$-entry to zero, the resulting pattern does not have a unique sepr, as it has $\SEPR(\sP)=\{\Sb\Sb\NN\NN, \Sb\Sb\Sp\NN,\Sb\Sb\Sm\NN\}$.
\end{ex}

\begin{prop}\label{prop:Aunique}  Let $\sP$ be an $n\x n$ sign pattern.  
If $t_1(\sP)=\Ap$ and $\sP$ has a unique sepr-sequence, then every signed  cycle product of  $\D(\sP)$ is positive.  If $t_1(\sP)=\Am$ and $\sP$ has a unique sepr-sequence, then every signed cycle product of $\D(\sP)$ of order $k$ has the sign $(-)^k$. 
\end{prop}
\bpf 
Suppose $t_1(\sP)=\Ap$ and $\sP$ has a unique sepr-sequence.  If  $\det\sP[\alpha]\ne+$, then $\sP[\alpha]$ is ambiguous (because it has a positive product of diagonal entries).  Then we see that  $\sP$ is ambiguous by considering  signed cycle products of opposite signs in $\alpha$ multiplied by positive diagonal entries not in $\alpha$. But having $\sP$  ambiguous contradicts $\sP$ having a unique sepr-sequence.  The proof for $t_1(\sP)=\Am$ is similar. 
\epf

The next example exhibits $\Ap\OL{\Ap}$ and $\Am\Ap\Am\cdots$ as unique sepr-sequences for irreducible sign patterns.
\begin{ex}
Let 
\[\sP^+=\begin{bmatrix}
+ & - & + & - & \cdots \\
+ & + & 0 & \cdots & 0 \\
0 & + & + & \ddots & \vdots \\
\vdots & \ddots & \ddots & \ddots & 0 \\
0 & \cdots & 0 & + & +
\end{bmatrix}\text{ and }
\sP^-=\begin{bmatrix}
- & - & - & - & \cdots \\
+ & - & 0 & \cdots & 0 \\
0 & + & - & \ddots & \vdots \\
\vdots & \ddots & \ddots & \ddots & 0 \\
0 & \cdots & 0 & + & -
\end{bmatrix}.
\]
Here the first row of  $\sP^+$  is sign-alternating and the first row of  $\sP^-$ is all $-$.  Then $\sepr(\sP^+)=\Ap\OL{\Ap}$ and $\sepr(\sP^-)=\Am\Ap\Am\cdots$.
\end{ex}

\subsection{Sufficient conditions  for a unique sepr-sequence}\label{ss:nec-suf-unique}

The next result follows immediately from Proposition \ref{prop:fix-k}.

\begin{cor}\label{prop:unique}
Let $\sP$ be an $n\times n$ sign pattern  such that for each $k=1,\ldots, n$, either every $k\times k$ principal subpattern has a signed determinant, or there are three $k\times k$ principal subpatterns that have signed determinants equal to $+$, $-$, and $0$, respectively.
Then $\sP$ has a unique sepr-sequence.
\end{cor}

Each of the sign patterns  in the next corollary has a unique sepr-sequence since  all principal subpatterns have  signed determinants.
\begin{cor}\label{cor:unique}
 Any sign pattern $\sP$ that has one of the following signed digraphs has a  unique sepr-sequence.  If an sepr-sequence is listed, it is determined by the signed digraph.
\ben[{\rm 1.}]
\item An $n$-cycle: $\overline{\NN}{\tt A}^{x}$, where $x$ is the sign of the signed cycle product of the $n$-cycle.
\item\label{dircycloops} An $n$-cycle with $1\le \ell \le n-1$ loops:  $t_k(\sP)\in\{\Sp,\Sm,\Sb\}$ for $1\le k\le \ell$, $t_k(\sP)=\NN$ for $\ell<k<n$  and $t_n(\sP)\in\{\Ap,\Am\}$. 
\item\label{dircycloops-all} An $n$-cycle with  $n$ loops such that the signed $n$-cycle product has the same sign as the signed product of the $n$ loops:  $t_k(\sP)\in\{\Ap,\Am,\Ab\}$ for $1\le k\le n-1$  and $t_n(\sP)\in\{\Ap,\Am\}$. 
\item A loopless skew-symmetric (doubly directed) cycle on $2s$ vertices with the $2s$-cycle product negative: $\overline{\NN\Sp}\NN\Ap$.   
\item A loopless symmetric (doubly directed) cycle on $2s+1$ vertices: Let $x$ denote the product of the signs of one $(2s+1)$-cycle.  $\OL{\NN\Sm\NN\Sp}\NN\Sm\NN\Ap{\tt A}^{x}$ for $s$ even and $\OL{\NN\Sm\NN\Sp}\NN\Am{\tt A}^{x}$ for $s$ odd. 
\item {A strong ditree with at most one loop.}
\een
\end{cor}

{If $\sP$ has a nonzero signed determinant, then $\sP$ is {sign nonsingular} \cite[Section 1.2]{BS95}.  Under the assumption that every diagonal entry of $\sP$ is negative, $\sP$ is sign nonsingular if and only if every cycle in $\Gamma(\sP)$ has a negative  cycle product \cite[Theorem 3.2.1]{BS95}.  Note that sign nonsingularity is preserved under permutation and signature multiplication, but these operations may change an sepr-sequence.}

\begin{cor}\label{cor:SNS} If $\sP$ is a sign pattern such that $\D(\sP)$ has each  cycle product  negative, then $\sP$ has a unique sepr-sequence.  This includes any sign nonsingular pattern $\sQ$ with all diagonal entries negative, which has $\sepr(\sQ)=\Am\Ap\OL{\Am\Ap}$ for $n(\sQ)$ even and $\sepr(\sQ)=\Am\OL{\Ap\Am}$ for $n(\sQ)$ odd.
\end{cor}

We conjecture that the converse of Corollary \ref{prop:unique} is also true.
\begin{conj}
\label{uniqueconj}
Let $\sP$ be an $n\times n$ sign pattern.  The following are equivalent:
\begin{enumerate}[{\rm (1)}]
\item\label{unique1} $\sP$ has a unique sepr-sequence;
\item\label{unique2} for each $k=1,\ldots, n$, either every $k\times k$ principal subpattern has  a signed determinant, or $t_k(\sP)=\Sb$ and there are three $k\times k$ principal subpatterns that  have signed determinants equal to  $+$, $-$, and $0$, respectively.
\end{enumerate} 
\end{conj}

Corollary \ref{prop:unique} established \eqref{unique2} implies \eqref{unique1}.   
We can  prove that \eqref{unique1} and an additional hypothesis imply \eqref{unique2} (Proposition \ref{oneamb}), \eqref{unique1}  implies parts of \eqref{unique2} (Proposition \ref{partialunique2}), and \eqref{unique1} implies \eqref{unique2} for sign patterns of order at most four (Proposition \ref{unique4}).

\begin{prop} 
\label{oneamb}
Suppose $\sP$ has a unique sepr-sequence, and for each $k=2,\dots,n-1$ there is at most one  $k\x k$ principal subpattern that does not have a signed determinant.  Then   $\sP$ satisfies \eqref{unique2} in Conjecture~\ref{uniqueconj}.
\end{prop}
\bpf
Every $1\x 1$ principal subpattern has a signed determinant. Since it is not possible to have $t_n(\sP)=\Sb$, $\sP$ has a signed determinant.  

Fix $k \in \{2,\ldots ,n-1\}$.  If every $k\times k$ principal subpattern has a signed determinant, then there is nothing to prove, so assume there is {exactly one} $k\times k$ principal subpattern $\sP[\alpha]$ that is ambiguous.  By Remark~\ref{rem:Sstar}, $t_k(\sP)=\Sb$.
Suppose first that  none of the rest of the $k\times k$ principal subpatterns has a signed determinant equal to $+$.   Then there exists a matrix $B_-\in Q(\sP)$ such that $\det B_-[\alpha]<0$. Then  $t_k(B_-)\in\{\Am,\Sm\}$, contradicting $t_k(\sP)=\Sb$.   Therefore, there must be a $k\times k$ principal subpattern that has a signed determinant equal to $+$.  Similarly, there must be one  $k\times k$ principal subpattern with a signed determinant equal to $-$ and one with a signed determinant equal to $0$.  \epf

\begin{lem}
\label{allnonzero} 
Let $\sP$ be an $n\x n$ sign pattern.  Then there exists a matrix $B\in Q(\sP)$ such that $\det B[\alpha,\beta]\ne 0$ for every $\alpha,\beta\subseteq[n]$ such that $|\alpha|=|\beta|$ and $\sP[\alpha,\beta]$ is ambiguous.
\end{lem}
\begin{proof}
{
Pick a matrix $B\in Q(\sP)$.  We will perturb $B$ inductively for $k=1,\ldots,n$ such that every (not necessarily principal) square submatrix of $B$ that corresponds to an ambiguous induced subpattern of $\sP$ has a nonzero determinant.

Since there is no $1\times 1$ ambiguous induced subpattern of $\sP$, the statement is true for $k=1$.  Now suppose every ambiguous induced subpattern of order smaller than $k$ corresponds to a submatrix in $B$ with nonzero determinant.  We will perturb $B$ so that the statement is true for $k$, while every nonzero minor of $B$ remains nonzero.  

A  $k$-subset  pair $(\gamma,\delta)$ is a pair $\gamma, \delta\subseteq[n]$  such that $|\gamma|=|\delta|=k$ and  $\sP[\gamma,\delta]$ is ambiguous.  Choose an ordering of the $k$-subset  pairs and let $(\alpha,\beta)$ be the next pair (meaning all prior pairs $(\gamma,\delta)$ have $\det B[\gamma,\delta]\ne 0$).  Suppose $\det B[\alpha,\beta]=0$.    Since $\sP[\alpha,\beta]$ is ambiguous, there is a nonzero term in $\det \sP[\alpha,\beta]$. Choose $u\in\alpha$ and $v\in\beta$ such that $p_{uv}\ne 0$ appears in a nonzero term.  Then  $\sP[\alpha\setminus\{u\},\beta\setminus\{v\}]$ is either ambiguous or has nonzero determinant.   If $\sP[\alpha\setminus\{u\},\beta\setminus\{v\}]$ is ambiguous, then $\det(B[\alpha\setminus\{u\},\beta\setminus\{v\}])\neq 0$ by hypothesis.  Therefore, for a sufficiently small perturbation of the $u,v$-entry of $B$, the determinant of $B[\alpha,\beta]$ becomes nonzero while all previously determined nonzero minors of $B$ remain nonzero.

Applying  this process from $k=2$ to $k=n$  through all possible $k$-subset  pairs $(\gamma,\delta)$ in order, the desired result follows.
}
\end{proof}
  
\begin{prop}
\label{partialunique2} 
Suppose $\sP$ has a unique sepr-sequence. Then for each $k=1,\ldots, n$, either every $k\times k$ principal subpattern has a signed determinant, or $t_k(\sP)=\Sb$ and there exists a $k\times k$ principal subpattern that has a signed determinant  zero.
\end{prop}
\bpf Suppose  $\sP[\beta]$ is ambiguous and $|\beta|=k$. 
Since $\sP$ has a unique sepr-sequence,  $t_k(\sP)=\Sb$. By Lemma \ref{allnonzero}, we can choose $B\in Q(\sP)$ such that $\det B[\alpha]\ne 0$ for every $\alpha$ with $\sP[\alpha]$ ambiguous.  In order to have $t_k(A)=\Sb$, there must be a set $\gamma$ such that $|\gamma|=k$ and $\det B[\gamma]=0$. If $\sP[\gamma]$ did not have a signed determinant, then $\det B[\gamma]\ne 0$  by the way $B$ was chosen.  So $\det B[\gamma]=0$ implies that $\sP[\gamma]$ has  determinant  zero.
\epf

\begin{lem}
\label{t2unique}
If $\sP$ is a sign pattern such that $\sP$ has a fixed $2$nd sepr term and $t_2(\sP)=\Sb$, then there are three $2\times 2$ principal subpatterns that have signed determinants $+$, $-$, and $0$, respectively.
\end{lem}
\begin{proof}
If every $2\times 2$ principal subpattern of $\sP$ has a signed determinant, then we are done, so assume that $\sP$ has some $2\times 2$ ambiguous principal subpattern.

Observe that every $2\times 2$ ambiguous principal subpattern  has all entries nonzero, and there are an even number of $+$ and an even number of $-$ entries.  Note that  
\[\det\begin{bmatrix}a&b\\c&d\end{bmatrix}=ad-bc,\]
so for any given $a,d$, there are $b,c$ of the correct signs to make the determinant positive, zero, or negative.  This means the off-diagonal entries of a $2\times 2$ ambiguous principal subpattern are enough to realize matrices with determinants $+, -, 0$.

Since the off-diagonal entries of different $2\times 2$ principal subpatterns never overlap, we can find a matrix $B$ such that every $2\times 2$ submatrix corresponding to an ambiguous principal subpattern has zero determinant.  Therefore, there exist two $2\times 2$ principal subpatterns of $\sP$ whose determinants are positive and negative, respectively.  Since Proposition~\ref{partialunique2} guarantees the existence of a $2\x 2$ pattern with zero determinant, this completes the proof.
\end{proof}

{\begin{prop}\label{unique4}
Conjecture~\ref{uniqueconj} is true for sign patterns of order $\leq 4$.
\end{prop}
\begin{proof}
By Corollary~\ref{prop:unique}, it is sufficient to show that (\ref{unique1}) implies (\ref{unique2}).  Suppose $\sP=\begin{bmatrix} p_{ij}\end{bmatrix}$ is an $n\times n$ sign pattern that has a unique sepr-sequence.

Statement \eqref{unique2} is always true for $k=1$, true for $k=2$ by Lemma~\ref{t2unique}; it is also true for $k=n$, for otherwise $\sP$ allows $t_n$ to be $\Ap$, $\Am$, and $\NN$.  Therefore, the conjecture is true when $n\leq 3$.  

Suppose $n=4$ and there is a $3\times 3$ ambiguous principal subpattern, which implies $t_3(\sP)=\Sb$.  Without loss of generality assume   $\sP(1)$ is ambiguous.  If $\sP(1)$ is the only $3\times 3$ ambiguous principal subpattern, then the desired result follows by Proposition~\ref{oneamb}, so assume there is another $3\times 3$ ambiguous principal subpattern, say $\sP(2)$ is ambiguous.  By Proposition~\ref{partialunique2} and a suitable relabeling, we may assume $\det(\sP(3))=0$.  Since $\sP(1)$ is ambiguous yet $\sP$ is not, it follows that $p_{1,1}=0$.  Similarly, $p_{2,2}=0$.

Suppose $p_{2,3}=0$.  Then $p_{2,4}$ is the only nonzero entry in the first row of $\sP(1)$, so $\sP(1)$ being ambiguous implies $\sP[\{3,4\},\{2,3\}]$ is ambiguous, implying all four entries are nonzero.  Since $\sP(2)$ is ambiguous, pick a matrix $B\in Q(\sP)$ such that $\det(A(2))=0$.  When every entry in $B(2)$ is fixed, we can still use the $3,2$-entry and the $4,2$-entry to make $\det(A(1))=0$.  Now $\det(A(1))=\det(A(2))=\det(A(3))=0$, so it is impossible to have $t_3(A)=\Sb$, which is a contradiction.  Therefore, $p_{2,3}\neq 0$.  Similarly, $p_{3,2}$, $p_{1,3}$, $p_{3,1}$, $p_{2,4}$, $p_{4,2}$, $p_{1,4}$, and $p_{4,1}$ are nonzero.

Since $\det(\sP(3))=0$, it follows that $p_{1,2}=p_{2,1}=0$, for otherwise $p_{1,2}p_{2,4}p_{4,1}$ or $p_{2,1}p_{1,4}p_{4,2}$ is nonzero.  Consequently, $\det(\sP(4))=0$.  By taking a matrix $B\in Q(\sP)$ with $\det(A(1))=0$, three of the four principal minors of $B$ of order $3$ are zero, so $t_3(A)\ne \Sb$, which is a contradiction.  Thus, the conjecture is true when $n=4$.
\end{proof} 


\subsection{Uniquely attainable sepr-sequences}\label{ss-uniqueattain}
We now determine conditions on an sepr-sequence for it to be attainable as the unique sepr-sequence of a sign pattern.

For a graph $G=(V,E)$ and a vertex $v\in V$, a {\em neighbor} of $v$ in $G$ is a vertex $u$ such that $\{v,u\}\in E$.
For $W\subseteq V$, $N_{G}(W)$ denotes the set of neighbors of vertices in $W$.
A graph $G=(V,E)$ is {\em bipartite} if the vertices can be partitioned as $X\dunion Y$ such that each edge of $G$ has one vertex in $X$ and one in $Y$.  The next result is well-known in graph theory.

\begin{thm}[Hall's Theorem]{\rm \cite[Theorem 2.1.2]{Diestel}} Let $G=(X\dunion Y,E)$ be a bipartite graph.
Then $G$ contains a matching of $X$ if and only if $|S|\le |N_G(S)|$ for all $S\subseteq X$.
\end{thm}

Let $\sP=\begin{bmatrix}p_{ij}\end{bmatrix}$ be an $n\x n$ sign pattern.  The {\em bigraph of $\sP$} is the bipartite graph with vertex set 
$X\dunion Y\text{, where }X=\{x_i\}_{i=1}^n, Y=\{y_i\}_{i=1}^n$ and edge set
$\{\{x_i,y_j\}: p_{ij}\neq 0\};$
the bigraph of $\sP$ is denoted by  $BG(\sP)$.  For $\alpha\subseteq [n]$, let $X_\alpha=\{x_i:i\in\alpha\}$, $Y_\alpha=\{y_i:i\in\alpha\}$, and  $BG(\sP)_\alpha=BG(\sP)[X_\alpha\dunion Y_\alpha]$.    Observe that $BG(\sP)_\alpha=BG(\sP[\alpha])$.  Note that a matching of $X_\alpha$ in $BG(\sP)_\alpha$ is a perfect matching of  $BG(\sP[\alpha])$.   

\begin{rem}\label{rem:match-det} A perfect matching  in the bigraph of an $n\x n$ sign pattern $\sP$ is naturally associated with a permutation of $[n]$ for which all the corresponding entries of $\sP$ are nonzero, and thus with a nonzero term in the determinant of $\sP$:  Suppose $M$ is a perfect matching of $BG(\sP)$.  For $i\in [n]$, define $\sigma(i)$ by $\{x_i,y_{\sigma(i)}\}\in M$.  Because $M$ is a perfect  matching, $\sigma$ is a permutation of $[n]$, and $p_{i\sigma(i)}\ne 0$ by the definition of $BG(\sP)$; thus $\det\sP\ne 0$. Any permutation $\sigma$ of $[n]$ such that $p_{1\sigma(1)}\cdots p_{n\sigma(n)}\ne 0$ yields a perfect matching $M=\{\{x_1,y_{\sigma(1)}\},\dots \{x_n,y_{\sigma(n)}\} \}$.  Thus $\det\sP\ne 0$ if and only if  $BG(\sP)$ has a perfect matching.
\end{rem} 
\begin{thm}\label{thm:Aforever} Let $\sP$ be a sign pattern that has a unique sepr-sequence and contains no ambiguous principal subpattern.  Then $t_k(\sP)\in\{\Ap,\Am,\Ab\}$ implies $t_{k+1}(\sP)\in\{\Ap,\Am,\Ab\}$.
\end{thm}
\begin{proof}
 Suppose $t_k(\sP)\in\{\Ap,\Am,\Ab\}$.  
Fix a subset $\gamma\subseteq [n]$ with $|\gamma|=k+1$.  We  show that $\det \sP[\gamma]\neq 0$ by showing there is a perfect matching on $BG(\sP)_\gamma$ and applying Remark \ref{rem:match-det}.  Since $t_k(\sP)\in\{\Ap,\Am,\Ab\}$, there is a perfect matching between $X_\alpha$ and $Y_\alpha$ in $BG(\sP)_\alpha$ for any $\alpha$ with $|\alpha|=k$.  For any $\beta\subset\gamma$ with $|\beta|\leq k$, $\beta$ is a subset of some $\alpha\subset\gamma$ with $|\alpha|=k$.  Since there is a perfect matching between $X_\alpha$ and $Y_\alpha$, it follows that $|X_\beta|\leq |N_{BG(\sP)_\gamma}(X_\beta)|$ for any $\beta$ with $|\beta|\leq k$.  Finally, every vertex in $Y_\gamma$ has at least one neighbor in $X_\gamma$, for otherwise $\sP[\gamma]$ has a zero column $j$ and $\det \sP[\alpha]=0$ for any $\alpha$ with $j\in\alpha$ and $|\alpha|=k$, violating the fact that $t_k\in\{\Ap,\Am,\Ab\}$.  Therefore, when $\beta=\gamma$, $N_{BG(\sP)_\gamma}(X_\gamma)=Y_\gamma$ and $|X_\gamma|=|N_{BG(\sP)_\gamma}(X_\gamma)|$.  By Hall's Theorem, there is a perfect matching on $BG(\sP)_\gamma$.  Since $\sP[\gamma]$ is not ambiguous by hypothesis, $\sP[\gamma]$ has a nonzero signed determinant.  
Since the argument holds for any $\gamma$ with $|\gamma|=k+1$, it follows that $t_{k+1}\in\{\Ap,\Am,\Ab\}$.
\end{proof}

\begin{prop}  \label{prop:noNA}
Let $\sP$ be an $n\times n$ sign pattern with $n\geq 5$ that has a signed determinant, or an $n\times n$ sign pattern with $n=4$ that has a unique sepr-sequence.  Then the first two terms of any sepr-sequence in $\SEPR(\sP)$ cannot be $\NN\Ap$, $\NN\Am$, or $\NN\Ab$.
\end{prop}
\begin{proof}
Suppose to the contrary that $\sP$ is a sign pattern that has a signed determinant and a realization $B$ such that $\sepr(B)$ starts with $\NN\Ap$, $\NN\Am$, or $\NN\Ab$.  Then the diagonal entries of $\sP$ are all zero, and the off-diagonal entries are all nonzero.  Thus $\sP$ has $n^2-n$ nonzero entries.  Since $\Gamma(\sP)$  has an $n$-cycle, the  signed determinant of $\sP$ is nonzero; thus, $\sP$ is sign nonsingular.  By \cite[Theorem 8.1.1]{BS95}, any sign nonsingular pattern has at most $\frac{n^2+3n-2}{2}$ nonzero entries.  However, 
\[n^2-n>\frac{n^2+3n-2}{2}\]
whenever $n\geq 5$, which is a contradiction.  

Now assume  $n=4$ and $\sP$ has a unique sepr-sequence, so $\sP$ has a signed determinant.  Thus $\sP$ has $12$ nonzero entries.  By Proposition~\ref{unique4}, there is at least one $3\times 3$ principal subpattern that has a signed determinant, say $\sP[\{2,3,4\}]$.  Since all off-diagonal entries are nonzero, $\sP[\{2,3,4\}]$ has a nonzero signed determinant.  Construct a sign pattern $\sP'$ from $\sP$ by changing the $1,1$-entry to be $p_{1,1}'\neq 0$ such that 
\[p_{1,1}'\cdot \det\sP[\{2,3,4\}] = \det \sP.\]
Any order 4 composite cycle of $\Gamma(\sP')$ using the loop on vertex $1$ includes a $3$-cycle on $\{2,3,4\}$. 
{Therefore, $\sP'$ is a sign nonsingular pattern with $13$ nonzero entries, which equals the upper bound $13=\frac{4^2+3\cdot 4-2}{2}$.}  By \cite[Theorem 8.1.1]{BS95}, there are two permutation matrices $Q_1$ and $Q_2$ such that the $i,j$-entry of $Q_1\sP'Q_2$ is zero whenever $i\geq j+2$, so the first column of $Q_1\sP'Q_2$ has exactly $2$ zero entries.  This is impossible because each column of $\sP'$ has at most one zero entry. 

Therefore, $\sepr(\sP)$ cannot start with $\NN\Ap$, $\NN\Am$, or $\NN\Ab$.
\end{proof}

For each of $\NN\Ap$, $\NN\Am$, or $\NN\Ab$, there is an $n\times n$ sign pattern with $n=2$ or $3$ that  starts with the specified initial pair and has a unique sepr-sequence, except that $\NN\Ab$ is impossible for $n=2$ and $\NN\Ap$ is impossible for $n=3$.  For example,  $\sepr\left(\mtx{0 & + \\- & 0}\right)=\NN\Ap$, $\sepr\left(\mtx{0 & + \\+ & 0}\right)=\NN\Am$,  $\sepr\left(\mtx{0 & + & +\\+ & 0 & +\\+ &+& 0}\right)=\NN\Am\Ap$, and $\sepr\left(\mtx{0 & - & -\\+ & 0 & +\\+ &+& 0}\right)=\NN\Ab\Am$.  For $n=3$, every sign pattern with $t_1=\NN$ and $t_2=\Ap$ is a skew-symmetric doubly directed $3$-cycle, which has ambiguous determinant.


\subsection{Sign semi-stable and other patterns without long cycles}\label{s-sss}

A matrix is {\em semi-stable} (respectively, {\em stable}) if {each of its eigenvalues has nonpositive (respectively, negative) real part}.  
A sign pattern is {\em sign semi-stable} (respectively, {\em sign stable}) if every matrix $B\in Q(\sP)$ is   semi-stable (respectively,  stable).

Recall (e.g., \cite[Fact 42.4.2]{HLA2SP}) that an $n\x n$ sign pattern $\sP=[p_{ij}]$ is sign semi-stable if and only if  
\bit
\item[($\alpha$)] $p_{ii}=-$ or $p_{ii}=0$ for $i=1,\dots,n$,
\item[($\beta$)]  $p_{ij}p_{ji}=-$ or $p_{ij}p_{ji}=0$ for $1\le i<j\le n$, and
\item[($\gamma$)] Any cycle  of $\D(\sP)$ has length at most two.
\eit
The next four statements follow from 
this characterization of sign semi-stable patterns.  
\ben[(i)]
\item\label{ssc0} Every strong component of a digraph of a  sign semi-stable pattern is a strong ditree.
\item \label{ssc01} The digraph of the simplified pattern   of a sign semi-stable pattern  is a strong diforest. 
\item\label{ssc1} A sign semi-stable pattern of order $k$ has a signed determinant equal to zero or  $(-)^k$. 
\item\label{ssc2} Every principal subpattern of
a sign semi-stable pattern is  sign semi-stable.
\een

We begin by establishing some necessary properties of a sign pattern $\sP$ such that every principal subpattern has a signed determinant and $\D(\sP)$ does not have cycles of length three or more.  This includes sign semi-stable patterns by \eqref{ssc1}, \eqref{ssc2}, and  \eqref{ssc01} above. 

\begin{thm}\label{thm:sss} Let $\sP$ be an $n\times n$ sign pattern such that every principal subpattern has a signed determinant and $\D(\sP)$ does not have cycles of length three or more. 
\ben[{\rm (a)}]
\item\label{sss1} $\sP$ has a unique sepr-sequence. 
\item\label{sss3} If  $t_k(\sP)\in\{\Ap,\Am,\Ab\}$, then
$t_{\ell}(\sP)\in\{\Ap,\Am,\Ab\}$ for $\ell=k,\dots,n$.
\item\label{sss34} If $t_k(\sP)=\NN$, then $t_{k+2}(\sP)=\NN$; moreover, if $k$ is even, then $t_j(\sP)=\NN$ for any $j\geq k$.
\item\label{sss4} Suppose  $t_1(\sP)=\NN$ and let $\sQ$ be the simplified pattern of $\sP$. Then, $t_k(\sP)\in\{\Sp,\Sm,\Sb\}$ if   $k<n$, $k$ is  even, and $k\le 2\match(\D(\sQ))$; $t_n(\sP)\in\{\Ap,\Am\}$ if $\D(\sQ)$ has a perfect matching; and $t_k(\sP)=\NN$ for $k$ odd or $k> 2\match(\D(\sQ))$.
\een
 \end{thm}

\bpf  Statements \eqref{sss1}  and \eqref{sss3}  are immediate from Corollary \ref{prop:unique} and Theorem \ref{thm:Aforever}, respectively.
By hypothesis,  the simplified pattern $\sQ$  of $\sP$  is a strong diforest.  By Remark \ref{rem:simp}, $\sepr(\sQ)=\sepr(\sP)$, so we work with $t_k(\sQ)$ for the rest of the proof.

Suppose $t_k(\sQ)=\NN$.  If $t_{k+2}(\sQ)\neq \NN$, then there is a composite cycle of order $k+2$.  Assume that this composite cycle is composed of $a$ $2$-cycles and $b$ $1$-cycles, so $2a+b=k+2$.  Since $t_k(\sQ)=\NN$, it is impossible that $b=k+2$.  Hence $a\geq 1$.  It follows that there is a composite cycle of order $k$ composed of $a-1$ $2$-cycles and $b$ $1$-cycles, a contradiction.  Therefore, $t_{k+2}=\NN$.  Now suppose $k$ is even. 
If $t_{k+1}(\sQ)\neq \NN$, then there is a composite cycle of order $k+1$, which is composed of $a$ $2$-cycles and $b$ $1$-cycles.  Since $2a+b=k+1$ is an odd number, $b\geq 1$.  Then there is a composite cycle of order $k$ composed of $a$ $2$-cycles and $b-1$ $1$-cycles, violating the fact that $t_k(\sQ)=\NN$.  Therefore, $t_{k+1}(\sQ)=t_{k+2}(\sQ)=\NN$ when $k$ is even.  Inductively, $t_j(\sQ)=\NN$ for all $j\geq k$.

Suppose $t_1(\sQ)=\NN$.  Then $t_k(\sQ)=\NN$ for any odd $k$.  If $k$ is even and $k> 2\match(\D(\sQ))$, then there is no composite cycle of order $k$, so $t_k(\sQ)=\NN$.  For any even $k$ with $2\leq k\leq 2\match(\D(\sQ))$, there is at least one composite cycle of order $k$, so $t_k(\sQ)\in\{\Sp,\Sm,\Sb,\Ap,\Am,\Ab\}$.  In any strong diforest with $2\le k< n$ and $k$ even, choosing the index of an isolated vertex or a leaf without its neighbor results in a combinatorially singular principal subpattern of order $k$. Thus, $t_k(\sQ)\in\{\Sp,\Sm,\Sb\}$ for all such $k$. If $n=2\match(\D(\sQ))$, then $t_n(\sQ)\in\{\Ap,\Am\}$.
 \epf

The next result follows from Theorem \ref{thm:sss} and by properties \eqref{ssc1} and \eqref{ssc2} of sign semi-stable patterns. 

\begin{cor}\label{cor:sss} Let $\sP$ be an $n\times n$ sign semi-stable pattern. 
\ben[{\rm (a)}]
\item\label{c:sss1} $\sP$ has a unique sepr-sequence  in which  $t_k(\sP)\in\{\NN,\,\Sp,\, \Ap\}$ if $k$ is even, and  $t_k(\sP)\in\{\NN,\,\Sm,\, \Am\}$ if $k$ is odd. 
\item\label{c:sss3} If  $t_k(\sP)\in\{\Ap,\Am\}$, then
$t_{\ell}(\sP)\in\{\Ap,\Am\}$ for $\ell=k,\dots,n$ with $\Ap$ and $\Am$ alternating.

\item\label{c:sss34} If $t_k(\sP)=\NN$, then $t_{k+2}(\sP)=\NN$; moreover, if $k$ is even, then $t_j(\sP)=\NN$ for any $j\geq k$.
\item\label{c:sss4} Suppose  $t_1(\sP)=\NN$ and let $\sQ$ be the simplified pattern of $\sP$. Then, $t_k(\sP)=\Sp$ if   $k<n$  even with $k\le 2\match(\D(\sQ))$, $t_n(\sP)=\Ap$ if $\D(\sQ)$ has a perfect matching, and $t_k=\NN$ for $k$ odd or $k> 2\match(\D(\sQ))$.
\een
 \end{cor}

Note that there exist sign patterns satisfying \eqref{c:sss1} -- \eqref{c:sss4} of Corollary \ref{cor:sss} that are not  sign semi-stable, as in the next example.
\begin{ex} Let   
$\sP=\mtx{0 & - & 0 & 0\\ + & 0 & + & 0\\ 0 & 0 & 0 & +\\+ & 0 & 0 & 0}$ and $\sQ=\mtx{0 & - & 0 & 0\\ + & 0 & - & 0\\ 0 & + & 0 & -\\0 & 0 & + & 0}$.  Then $\sepr(\sP)=\sepr(\sQ)=\NN\Sp\NN\Ap$, $\sQ$ is sign semi-stable, but  $\sP$ is not sign semi-stable (because $\Gamma(\sP)$ has a $4$-cycle). \end{ex}

Next we characterize sepr-sequences that can be uniquely realized only by sign semi-stable patterns. 

\begin{lem}\label{lem:semidoable}
Suppose $\sP$ is an $n\times n$ sign pattern that has a unique sepr-sequence  with   $t_1(\sP)\in\{\Am,\Sm,\NN\}$ and $t_2(\sP)\in\{\Ap,\Sp,\NN\}$.  
If $t_k(\sP)=\NN$ for all $k\geq 3$ or $n\leq 2$, then $\sP$ is sign semi-stable.
\end{lem}
\begin{proof}
Let $\sP=\begin{bmatrix}p_{ij}\end{bmatrix}$.  Since $t_1(\sP)\in\{\Am,\Sm,\NN\}$, $p_{ii}=-$ or $p_{ii}=0$ for $i=1,\dots,n$.  Since $t_2(\sP)\in\{\Ap,\Sp,\NN\}$ and $\sP$ has a unique sepr-sequence, $p_{ij}p_{ji}=-$ or $p_{ij}p_{ji}=0$ for $1\le i<j\le n$.  The fact that $\sP$ has a unique sepr-sequence with $t_k(\sP)=\NN$ for all $k\geq 3$ implies $\Gamma(\sP)$ does not have any cycle of length $3$ or higher. Therefore,  $\sP$ is sign semi-stable.  
\end{proof}

As we will see in Theorem~\ref{thm:semirecog}, any sepr-sequence of a sign semi-stable pattern can be the unique sepr-sequence of another sign pattern that is not sign semi-stable, except for those sepr-sequences described in Lemma~\ref{lem:semidoable}.  Before proving Theorem~\ref{thm:semirecog}, we first study the sepr-sequence of some sign semi-stable patterns.

Let $P_k$ be the loopless doubly directed digraph whose underlying graph is a path on $k$ vertices.  Let $\Pl_k$, $\Pll_k$, and $P^\ell_k$ be the digraphs obtained from $P_k$ by adding a loop at one of the endpoints, both of the endpoints, or every vertex  of $P_k$, respectively.  Let $S_k$ be the loopless doubly directed digraph whose underlying graph is a star on $k$ vertices, and let $\Sl_k$ be the digraph obtained from $S_k$ by adding a loop on the center vertex.

\begin{lem}\label{lem:addcycle}
If $\sP$ is a simplified sign semi-stable pattern and $\Gamma(\sP)$ contains $P_4$ or $\Pl_3$ as a (not necessarily induced) subdigraph, then there is another sign pattern $\sP'$ with $\sepr(\sP')=\sepr(\sP)$ and $\sP'$ is not sign semi-stable.
\end{lem}
\begin{proof}
Suppose $\sP$ is a simplified sign semi-stable pattern such that $\Gamma(\sP)$ contains $P_4$ as a subdigraph.  Assume $P_4$ has vertices $u$, $v$, $x$, and $y$ in path order.  Obtain $\sP'$ from $\sP$ by changing the $y,u$-entry to be nonzero so that the $4$-cycle $C$  
thus created has a positive signed cycle product.  Thus, $\sP'$ is not a sign semi-stable pattern.  Since $\Gamma(\sP)$ is a strong diforest, the only cycle of $\Gamma(\sP')$ that contains the arc $(y,u)$ is $C$.  For any principal subpattern $\sP'[\alpha]$, if $\{u,v,x,y\}$ is not contained in $\alpha$, then $\det(\sP'[\alpha])=\det(\sP[\alpha])$; if $\{u,v,x,y\}\subseteq \alpha$, then any composite cycle on $\alpha$ using $C$ gives another composite cycle by replacing $C$ by two $2$-cycles with vertices $\{u, v\}$ and $\{x, y\}$, and they have the same sign, so $\det(\sP'[\alpha])=\det(\sP[\alpha])$.  Therefore, $\sepr(\sP')=\sepr(\sP)$.

Now suppose $\Gamma(\sP)$ contains a $\Pl_3$ on vertices $u$, $v$, and $x$, where $x$ is the vertex with a loop.  Then  obtain $\sP'$ from $\sP$ by changing the $x,u$-entry to be nonzero so that the $3$-cycle $C$ thus created has a negative signed cycle product.  Thus, whenever $C$ is in a composite cycle of $\Gamma(\sP')$, it yields another composite cycle on the same set of vertices with the same sign by replacing $C$ with the $2$-cycle on vertices $\{u,v\}$ and the loop on $x$.  Thus, $\sP'$ has the desired properties.
\end{proof}

\begin{lem}\label{prop:semistars}
Let $\sP$ be a $k\times k$ sign semi-stable pattern with $k\geq 2$. 
\begin{itemize}
\item If $\Gamma(\sP)=S_k$, then $\sepr(\sP)=\sepr\begin{bmatrix}0&+\\-&0\end{bmatrix}\oplus O_{k-2}$, and the sepr-sequence is $\NN\Sp\NN\OL{\NN}$ if $k\geq 3$ and $\NN\Ap$ if $k=2$.
\item If $\Gamma(\sP)=\Sl_k$, then $\sepr(\sP)=\sepr\begin{bmatrix}-&+\\-&0\end{bmatrix}\oplus O_{k-2}$, and the sepr-sequence is $\Sm\Sp\NN\OL{\NN}$ if $k\geq 3$ and $\Sm\Ap$ if $k=2$.
\end{itemize}
\end{lem}
\begin{proof}
There are no composite cycles of order $3$ or higher in the digraphs of all the above mentioned sign patterns, and using the tables in Section \ref{s:basic} it is straightforward to verify the values of $t_1(\sP)$ and $t_2(\sP)$.
\end{proof}

\begin{lem}\label{prop:semismall}
Let $\sP$ be a $k\times k$ sign semi-stable pattern.
\begin{itemize}
\item If\/ $\Gamma(\sP)$ is $P_2\dunion P_2$ or $P_4$, then $\sepr(\sP)=\NN\Sp\NN\Ap$.
\item If\/ $\Gamma(\sP)$ is $\Pl_2\dunion\Pl_2$, $\Pl_2\dunion P_2$, $\Pll_2\dunion\Pl_2$, $\Pll_2\dunion P_2$, or $\Pl_4$, then $\sepr(\sP)=\Sm\Sp\Sm\Ap$.
\item If\/ $\Gamma(\sP)$ is $\Pll_2\dunion\Pll_2$ or $P^\ell_4$, then $\sepr(\sP)=\Am\Ap\Am\Ap$. 
\item If\/ $\Gamma(\sP)$ is $\Pl_2\dunion\Pl_1$, $P_2\dunion \Pl_1$, or $\Pl_3$, then $\sepr(\sP)=\Sm\Sp\Am$.
\item If\/ $\Gamma(\sP)$ is $\Pll_2\dunion\Pl_1$ or $P^\ell_3$, then $\sepr(\sP)=\Am\Ap\Am$.
\end{itemize}
If\/ $\Gamma(\sP)=\Gamma_1\dunion\Gamma_2$ with $\Gamma_1\in\{P_2,\Pl_2,\Pll_2\}$ and $\Gamma_2\in\{P_2,\Pl_2,\Pll_2,\Pl_1\}$, then there is a sign pattern $\sP'$ such that $\sepr(\sP')=\sepr(\sP)$ and $\sP$ is not sign semi-stable.
\end{lem}
\begin{proof}
It is straightforward to verify the sepr-sequences for the listed small order $k$ sign semi-stable patterns.  Whenever $\Gamma$ is $P_4$, $\Pl_4$, $P^\ell_4$, $\Pl_3$, or $P^\ell_3$, it contains $P_4$ or $\Pl_3$ as a subdigraph, so the existence of $\sP'$ follows from Lemma~\ref{lem:addcycle}.
\end{proof}

\begin{thm}\label{thm:semirecog}
Suppose $\sP$ is  an $n\x n$ sign semi-stable pattern.  Then every sign pattern that has a unique sepr-sequence equal to $\sepr(\sP)$ is sign semi-stable if and only if $t_k(\sP)=\NN$ for $k\geq 3$ or $n\leq 2$.
\end{thm}
\begin{proof}
The sufficient condition follows from Lemma~\ref{lem:semidoable}.

Let $\sepr(\sP)=t_1\cdots t_n$.  Suppose every sign pattern with the unique sepr-sequence $t_1\cdots t_n$ is sign semi-stable.  Since the simplified sign pattern of $\sP$ has the same sepr-sequence, we may assume $\sP$ is simplified and $\Gamma=\Gamma(\sP)$ is a strong diforest.  By Lemma~\ref{lem:addcycle}, $\Gamma$ does not contain a $P_4$ or $\Pl_3$ as a subdigraph.  Thus, the underlying graph of $\Gamma$ is a disjoint union of stars.  
Since a star of order three or more with a loop on a  vertex that is not the center has a $\Pl_3$ subgraph, we may assume $\Gamma$ is a disjoint union of $a$ copies of $S_k$ (with possibly different $k$ values), $b$ copies of $\Sl_k$ (with possibly different $k$ values), $c$ copies of $\Pll_2$, $d$ copies of $P_1$, and $e$ copies of $\Pl_1$.  By Lemma~\ref{prop:semistars}, we may assume 
\[\Gamma=aP_2\dunion b\Pl_2\dunion\, c\!\Pll_2\dunion\, dP_1\dunion e\Pl_1.\]  
Since 
\[\sepr\begin{bmatrix}-&0\\0&-\end{bmatrix}=\sepr\begin{bmatrix}-&+\\-&-\end{bmatrix}=\Am\Ap,\]
we may trade two copies of $\Pl_1$ for one copy of $\Pll_2$ and assume $e\leq 1$.  
If $a+b+c+e\ge 2$, then by Lemma~\ref{prop:semismall} there would be a sign pattern $\sP'$ that is not sign semi-stable and $\sepr(\sP')=\sepr(\sP)$.
So  $a+b+c+e\leq 1$ and $\Gamma=\Gamma'\dunion dP_1$ with $\Gamma'\in \{P_2,\Pl_2,\Pll_2,\Pl_1\}$.  In all possible cases, $\Gamma$ does not contain any composite cycles of order $3$ or higher, so $t_k=\NN$ for $k\geq 3$ or $n\leq 2$.
\end{proof}

Recall (e.g., \cite[Fact 42.4.3]{HLA2SP}) that an irreducible $n\x n$ sign pattern $\sP=[p_{ij}]$ is sign stable if and only if  it is sign semi-stable and in addition
\bit
\item[($\delta$)] $\sP$ is not combinatorially singular, i.e., there is a nonzero term in the determinant of $\sP$.
\item[($\varepsilon$)]  There does not exist a nonempty subset $\beta$ of $[n]$ such that each diagonal
entry of $\sP[\beta]$ is zero, each row of $\sP[\beta]$ contains at least one nonzero entry, and
no row of $\sP[\OL{\beta},\beta]$ contains exactly one nonzero entry, where $\OL{\beta}=[n]\setminus\beta$.
\eit
  Unfortunately, sepr-sequences do not distinguish between sign stable and sign semi-stable patterns as the next example shows.
  
  \begin{ex}
	Let 
\[\sP=\begin{bmatrix}
- & + & 0 & 0 & 0 \\
- & 0 & + & 0 & 0 \\
0 & - & 0 & + & 0 \\
0 & 0 & - & 0 & + \\
0 & 0 & 0 & - & 0 \\
\end{bmatrix}\text{ and }\sQ=\begin{bmatrix}
0 & + & 0 & 0 & 0 \\
- & 0 & + & 0 & 0 \\
0 & - & - & + & 0 \\
0 & 0 & - & 0 & + \\
0 & 0 & 0 & - & 0 \\
\end{bmatrix}.\]
Both $\sP$ and $\sQ$ are sign semi-stable; $\sP$ is sign stable while $\sQ$ is not.  Since $\sepr(\sP)=\sepr(\sQ)=\Sm\Sp\Sm\Sp\Am$, the sepr-sequence is not able to distinguish a sign stable pattern from other sign semi-stable patterns.\end{ex}

An sepr-sequence satisfies the conclusion of condition \eqref{c:sss4} in Corollary \ref{cor:sss} if and only if it is of the form $\NN\,\OL{\Sp\NN}\,\OL{\NN}$  (any $n$) or $\NN\,\OL{\Sp\NN}\,{\Ap}$ (even $n$).   
\begin{prop}\label{sssN} Every sepr-sequence  $\NN\,\OL{\Sp\NN}\,\OL{\NN}$ or $\NN\,\OL{\Sp\NN}\,\Ap$  is attained by a sign semi-stable pattern $\sP$.
\end{prop}
\bpf Let $m$ be the maximum $k$ such that $t_k\in\{\Sp,\Ap\}$. Choose a strong ditree $\D$ such that $\D$ has no loops and $\match(\D)=\frac m 2$.  Let $\sP$ be any sign pattern with skew-symmetric off-diagonal part and digraph $\D$.
\epf

Case \eqref{c:sss4} in Corollary \ref{cor:sss}, which assumes $t_1=\NN$, is quite different from the more general situation in a variety of ways.  At most one {\tt A} is possible in the sepr-sequence  of a sign pattern  that satisfies conditions  \eqref{c:sss1} -- \eqref{c:sss4} of Corollary \ref{cor:sss} but Lemma \ref{prop:semismall} shows it is possible to have more than one ${\tt A}$ in the sequence if the diagonal is not all zero. 

Proposition \ref{sssN} shows that any  sepr-sequence of a sign pattern that satisfies conditions \eqref{c:sss1} -- \eqref{c:sss4} of Corollary \ref{cor:sss} can be attained by an irreducible sign semi-stable pattern.  However, the next proposition 
shows there are sepr-sequences that satisfy \eqref{sss1} -- \eqref{sss34} of Theorem \ref{cor:sss}  but are not attainable by any sign semi-stable pattern.

\begin{prop}\label{prop:sss-no}  An sepr-sequence $t_1\cdots t_n$ having $t_1\in\{\Sp,\Sm,\Sb\}$, $t_{n-1}=\NN$, and $t_n\in\{\Ap,\Am\}$ cannot be attained by any sign pattern $\sP$ such that every principal subpattern has signed determinant and $\D(\sP)$ has no cycle of order three or more.
\end{prop}
\bpf 
Let $\sP$ be a sign  pattern such that every principal subpattern has signed determinant, $\D(\sP)$ has no cycle of order three or more, and
   $t_{n-1}(\sP)=\NN$. If $n$ is odd, then $t_n(\sP)=\NN$ by Theorem \ref{thm:sss}\eqref{c:sss34}.  So assume $n$ is even.  If $t_n(\sP)\in\{\Ap,\Am\}$, then  there is a composite cycle of order $n$.  Since $\Gamma(\sP)$ has no cycle of length $3$ or higher, we may assume the composite cycle is composed of $a$ $2$-cycles and $b$ $1$-cycles with $2a+b=n$.  If $b\geq 1$, then there is a composite cycle of order $n-1$ and $t_{n-1}(\sP)\neq \NN$, so $t_{n-1}(\sP)=\NN$ implies $b=0$, $2a=n$, and there are disjoint $2$-cycles that cover the vertices of $\D(\sP)$.

Suppose $t_1(\sP)\in\{\Sp,\Sm,\Sb\}$.  Then $\D(\sP)$ has a loop, say it is on vertex $v$.  Take the composite cycle of order $n$, remove the $2$-cycle that covers $v$, and include the loop in the new composite cycle.  This yields a new composite cycle of order $n-1$, a contradiction.  Therefore, such an sepr-sequence is impossible for a sign pattern $\sP$ such that every principal subpattern has signed determinant and $\D(\sP)$ has no cycle of order three or more.
\epf

For each sepr-sequence that  Proposition \ref{prop:sss-no} establishes cannot be attained by a sign pattern $\sP$ such that every principal subpattern has signed determinant and $\D(\sP)$ has no cycle of order three or more, Corollary \ref{cor:unique}.\ref{dircycloops} provides an example of a sign pattern with an $n$-cycle that does realize the sequence.

A sign pattern $\sP$ is {\em signed cycle positive} if every signed cycle product in $\sP$ is positive. 
The next result is immediate from the definition and Theorems \ref{thm:Aforever} and \ref{thm:sss}.

\begin{cor}\label{cor:scp} Let $\sP$ be a signed cycle positive pattern. Then 
\ben[{\rm (1)}]

\item $\sP$ has a unique sepr-sequence in which each term is one of $\NN,\,\Sp,$ or $\Ap$. 
\item\label{scp3} If  $t_k(\sP)=\Ap$, then $t_{j}(\sP)=\Ap$ for any $j\geq k$.

\hspace{-6mm}If in addition $\D(\sP)$ has no cycle of order three or more, then
\item\label{scp34} If $t_k(\sP)=\NN$, then $t_{k+2}(\sP)=\NN$; moreover, if $k$ is even, then $t_j(\sP)=\NN$ for any $j\geq k$.
\item\label{scp4} Suppose  $t_1(\sP)=\NN$ and let $\sQ$ be the simplified pattern of $\sP$. Then, $t_k(\sP)=\Sp$ if   $k<n$  even with $k\le 2\match(\D(\sQ))$, $t_n(\sP)=\Ap$ if $\D(\sQ)$ has a perfect matching, and $t_k=\NN$ for $k$ odd or $k> 2\match(\D(\sQ))$.
\een
 \end{cor}


\section{Sepr-sequences for symmetric nonnegative matrices and sign patterns}\label{s:nnsym}

In this section we return to symmetric matrices, for which the study of sepr-sequences of matrices was introduced in \cite{Xavier17}, but focus on nonnegative matrices, and also study symmetric nonnegative sign patterns.   


\subsection{Sepr-sequences for symmetric nonnegative matrices}\label{ss:nnmtx}
The inverse eigenvalue problem for nonnegative matrices (NIEP) asks for all possible spectra of nonnegative matrices and has received much attention; see, e.g., \cite{HLA2IEP,Smigoc04} and the references therein.  Since the characteristic polynomial is enough to determine the spectrum, the principal minor assignment problem for nonnegative matrices is a refinement of the NIEP.  Understanding sepr-sequences for symmetric nonnegative matrices would be a useful step for these problems.

\begin{obs}
\label{abuden}
For a nonnegative matrix, the sepr-sequence must start with $\Ap$, $\NN$, or $\Sp$.
\end{obs}

Proposition~\ref{prop:forbsepr} collects some restrictions on sepr-sequences that can be attained by real symmetric matrices established in \cite{Xavier17}.  As noted in the beginning of Section \ref{s:basic}, matrix symmetry is necessary for many of these restrictions to apply.

\begin{prop}\label{prop:forbsepr}
Let $t_1\cdots t_n$ be the sepr-sequence of a Hermitian matrix.  Then the following cannot occur:
\begin{itemize}
\item $t_k=t_{k+1}=\NN$ but $t_j\neq \NN$ for some $j>k$. {\rm \cite[Theorem 2.3 ($\NN\NN$ Theorem)]{Xavier17}}
\item $\NN\Ab\cdots$, $\NN\Ap\cdots$, $\NN\Sb\cdots$, $\NN\Sp\cdots$, and $\Sp\Ap\cdots$.  {\rm \cite[Proposition 3.1]{Xavier17}}
\item $\cdots\Ab\NN\cdots$, $\cdots\NN\Ab\cdots$, and $\cdots\Sb\NN\NN\cdots$. {\rm \cite[Corollary 3.3 and Theorem 3.4]{Xavier17}} 
\end{itemize}
\end{prop}

Note that no sepr-sequence starting with $\NN\Ab$, $\NN\Ap$, $\NN\Sb$, $\NN\Sp$, or $\Sp\Ap$ can be attained by any  nonnegative matrix (symmetry is not required), since all the signed $2$-cycle products are negative. 

\begin{prop}\label{PropDiag}
If the sepr-sequence of a symmetric nonnegative matrix starts with $\Ap$, then there is a nonnegative matrix with the same sepr-sequence and all diagonal entries equal to one.
\end{prop}
\begin{proof}
Let $B=\begin{bmatrix} b_{ij}\end{bmatrix}$ be the given matrix.  Let $D=\diag(d_1,\ldots, d_n)$ be the diagonal matrix with $d_i=\frac{1}{\sqrt{b_{ii}}}$.  The matrix $DBD$ is symmetric, nonnegative, and all diagonal entries are equal to one. Also, every minor of $DBD$ has the same sign as the corresponding minor in $B$ because
\[\det((DBD)[\alpha,\beta])=\det(B[\alpha,\beta])\prod_{i\in\alpha}d_i\prod_{j\in\beta}d_j.\qedhere\]  
\end{proof}

The next result follows from applying Observation~\ref{abuden} and considering the examples in \cite[Table 1]{Xavier17}.

\begin{prop}
For $2\times 2$ matrices, an sepr-sequence is attainable by a symmetric nonnegative matrix if and only if the sepr-sequence is attainable by a Hermitian matrix and it starts with $\Sp$, $\Ap$, or $\NN$.
\end{prop}
\begin{proof}
By Observation~\ref{abuden}, an sepr-sequence is attainable by a symmetric nonnegative matrix only if the sepr-sequence is attainable by a Hermitian matrix and it starts with $\Sp$, $\Ap$, or $\NN$.  According to \cite[Table 1]{Xavier17}, the possible such sepr-sequences  are \[\Ap\Ap,\Ap\Am,\Ap\NN,\NN\Am,\NN\NN,\Sp\Am,\Sp\NN.\]
It is easy to check that the matrices given in \cite{Xavier17} that realize these sepr-sequences are nonnegative, so the necessary condition is also sufficient.
\end{proof}

The possible sepr-sequences of $3\times 3$ symmetric nonnegative matrices are listed in Table \ref {tab:order3}. In Example~\ref{ex:nns3} we provide symmetric nonnegative matrices that attain some sepr-sequences  where the realization given in \cite[Table 2]{Xavier17} is not nonnegative.

\begin{table}[b!]
\caption{All possible sepr-sequences of
order $n=3$
that are attainable by symmetric nonnegative matrices.  In the source  column, `Original' means  the original  matrix in \cite[Table 2]{Xavier17} is nonnegative (and its name is given in the second column); `New' means the matrix listed in \cite[Table 2]{Xavier17} is not nonnegative but we constructed a nonnegative matrix (the name appears in the second column and the matrix is given in Example \ref{ex:nns3}). 
}\label{tab:order3}
\begin{center}{\small
\renewcommand{\arraystretch}{1.3} \begin{tabular}{|l|l|l|}\hline
Sepr-sequence\qquad$\null$  & nonnegative matrix & source  \\ 
\hline
$\Ap\Ab\Am$ & $N_{\Ap\Ab\Am}$ &
New  \\[0.3mm]

\hline
$\Ap\Ap\Ap$ & $I_3$   &  Original  \\

\hline
$\Ap\Ap\Am$ & $N_{\Ap\Ap\Am}$    &
New \\[0.3mm]

\hline
$\Ap\Ap\NN$ & $N_{\Ap\Ap\NN}$   &
New  \\[0.3mm]

\hline
$\Ap\Am\Ap$ & $M_{\Ap\Am\Ap}$    &
Original  \\[0.3mm]

\hline
$\Ap\Am\Am$ & $N_{\Ap\Am\Am}$    &
New   \\[0.3mm]

\hline
$\Ap\Am\NN$ & $M_{\Ap\Am\NN}$ &
Original  \\[0.3mm]


\hline
$\Ap\NN\NN$ & $J_3$    &  Original  \\

\hline
$\Ap\Sb\Am$ & $N_{\Ap\Sb\Am}$  &
New   \\[0.3mm]

\hline
$\Ap\Sp\Am$& $M_{\Ap\Sp\Am}$   &
Original  \\

\hline
$\Ap\Sp\NN$ & $J_1 \oplus J_2$   &  Original  \\

\hline
$\Ap\Sm\Am$ & $M_{\Ap\Sm\Am}$   &
Original   \\

\hline
$\Ap\Sm\NN$ & $M_{\Ap\Sm\NN}$   &
Original  \\

\hline
\end{tabular}
\hfill
\begin{tabular}{|l|l|l|}\hline
Sepr-sequence\qquad$\null$  & nonnegative matrix & source  \\ 
\hline
$\NN\Am\Ap$ & $J_3 - I_3$   &  Original  \\



\hline
$\NN\NN\NN$     & $O_3$    &  Original \\

\hline
$\NN\Sm\NN$ & $M_{\NN\Sm\NN}$    &
Original  \\

\hline
$\Sp\Ab\Am$   &  $M_{\Sp\Ab\Am}$  &
Original  \\

\hline
$\Sp\Am\Ap$     &  $M_{\Sp\Am\Ap}$   &
Original  \\

\hline
$\Sp\Am\Am$    &  $M_{\Sp\Am\Am}$   &
Original  \\

\hline
$\Sp\Am\NN$ & $M_{\tt S^+A^-N}$    &
Original  \\

\hline
$\Sp\NN\NN$ & $J_1 \oplus O_2$   &
Original  \\

\hline

$\Sp\Sb\Am$ & $N_{\Sp\Sb\Am}$    &
New  \\

\hline
$\Sp\Sp\NN$     & $M_{\Sp\Sp\NN}$    &
Original  \\

\hline
$\Sp\Sm\Am$    & $M_{\tt S^+S^-A^-}$   &
Original  \\

\hline
$\Sp\Sm\NN$     & $M_{\Sp\Sm\NN}$    &
Original  \\

\hline

\end{tabular}}
\end{center}
\end{table}

\begin{ex}\label{ex:nns3}
Each of the following $3\times 3$ symmetric nonnegative matrices achieves the sepr-sequences listed in its subscript.  

\[N_{\Ap\Ab\Am}=\mtx{
1&2&0\\
2&1&2\\
0&2&1
}\qquad
N_{\Ap\Ap\Am}=\mtx{
1&0.9&0\\
0.9&1&0.9\\
0&0.9&1}\qquad
N_{\Ap\Ap\NN}=\mtx{
1&\frac{3}{5}&0\\
\frac{3}{5}&1&\frac{4}{5}\\
0&\frac{4}{5}&1
}\]
\[N_{\Ap\Am\Am}=\mtx{
1 & 8 & 2 \\
8 & 1 & 2 \\
2 & 2 & 1 \\
}\qquad
N_{\Ap\Sb\Am}=\mtx{
1 & 1 & 0 \\
1 & 1 & 2 \\
0 & 2 & 1}\qquad
N_{\Sp\Sb\Am}=\mtx{
0 & 1 & 0 \\
1 & 1 & 0 \\
0 & 0 & 1  \\
}\]
\end{ex}

The only sepr-sequences listed in   \cite[Table 2]{Xavier17} that start with $\Ap$, $\Sp$, or $\NN$ that are not listed in  Table~\ref{tab:order3} are $\Ap\NN\Am$, $\NN\Am\Am$, and $\NN\Am\NN$.  Proposition \ref{PropStart}, which gives some restrictions on the sepr-sequence of a symmetric nonnegative matrix, shows these sepr-sequences cannot be attained by a symmetric nonnegative matrix. Thus,  Table~\ref{tab:order3} gives a complete list of all sepr-sequences that are attainable by a symmetric nonnegative matrix of order $3$.

\begin{prop}\label{PropStart}
Consider the sepr-sequence of a symmetric nonnegative matrix of order $\geq 3$.
\bit
\item If the sepr-sequence  starts with $\Ap\NN$, then it must be $\Ap\NN\OL{\NN} \NN$.
\item If the sepr-sequence  starts with $\NN\Am$, then it starts with $\NN\Am\Ap$.
\item If the sepr-sequence  starts with $\NN\Sm$, then the third term is either $\Sp$ or $\NN$; in the case of order $3$, it must be $\NN\Sm\NN$.
\eit\end{prop}
\begin{proof}
If $B=\begin{bmatrix}b_{ij}\end{bmatrix}$ is a symmetric nonnegative matrix whose sepr-sequence starts with $\Ap\NN$, then we may assume all the diagonal entries are one by Proposition \ref{PropDiag}.  Since $t_2=\NN$, it follows that
\[b_{ii}b_{jj}-b_{ij}^2=0\]
for all $i,j$.  Since $b_{ii}=b_{jj}=1$ and $b_{ij}\geq 0$, we know $b_{ij}=1$ for any $i\neq j$.  Therefore, after scaling all diagonal entries to one, the new matrix (with the same sepr-sequence) is an all-ones matrix.  Consequently, the sepr-sequence is $\Ap\NN\OL{\NN} \NN$.  

If the sepr-sequence starts with $\NN\Am$, then all diagonal entries are zero and all off-diagonal entries are positive, so every order $3$ principal minor is positive.

Now let $B$ be a symmetric nonnegative matrix starting with $\NN\Sm$, so $B$ has zero diagonal.  Any order $3$ principal minor is either positive or zero: if the corresponding (doubly directed) digraph has a 3-cycle, i.e., if all off-diagonal entries are nonzero, then the determinant is always positive; if some off-diagonal entry is zero then the determinant is always zero.  Since the second term is $\Sm$, there must be a zero off-diagonal entry.  Therefore, at least one of the $3\times 3$ principal submatrices has a zero off-diagonal entry and so has determinant zero.  If one of the order $3$ principal minors has all off-diagonal entries  nonzero, then the third term is $\Sp$, otherwise it is $\NN$.
For order $3$, $t_3$ must be $\NN$ as $\Sp$ is not possible for the determinant.
\end{proof}

\begin{cor}
Consider the sepr-sequence $t_1t_2\cdots t_n$  of a symmetric nonnegative matrix with $t_1=\NN$ and $n\geq 3$.  Then the initial sequence $t_1t_2t_3$ must be one of  the following:
\[\NN\Am\Ap, \NN\Sm\NN, \NN\Sm\Sp, \NN\NN\NN.\]
\end{cor}
\begin{proof}
By Proposition~\ref{prop:forbsepr}, $t_2\notin\{\Ab,\Ap,\Sb,\Sp\}$, so $t_2\in\{\Am,\Sm,\NN\}$.  If $t_2=\NN$, then $t_k=\NN$ for all $k$ by Proposition~\ref{prop:forbsepr}.  The other cases follow from Proposition~\ref{PropStart}.
\end{proof}


\subsection{Sepr-sequences for symmetric nonnegative sign patterns}\label{ss:nnpat}

An $n\x n$ sign pattern $\sP=[p_{ij}]$ is {\em symmetric}  if $p_{ij}=p_{ji}$ for all $i,j=1,\dots,n$. In the spirit of Section \ref{s:unique}, we consider  the possible unique sepr-sequences of symmetric nonnegative sign patterns (without assuming the matrices described by $\sP$ are symmetric).     
We continue to use digraph terminology, because we need to track the presence or absence of loops, but note that the digraph of a symmetric pattern is necessarily doubly directed. For nonnegative patterns, we need not formally sign the digraph, because all arcs represent  positive entries.
Let $K_n$ be the loopless doubly directed graph whose underlying graph is a complete graph on $n$ vertices.  Let $\Kl_n$ be obtained from $K_n$ by adding a loop.    For $n\ge 3$,  the {\em leaf-loop-star}, denoted by $S_n^{(\ell)}$, is the doubly directed digraph whose underlying graph is a star of order $n$, and there is a loop on every vertex except the center vertex; $S_6^{(\ell)}$ is shown in Figure \ref{fig:leaf-loop-star}.  For any digraph $\Gamma$, $k\Gamma$ denotes the disjoint union of  $k$ copies of $\Gamma$.

\begin{figure}[h!] \begin{center}
\scalebox{.4}{\includegraphics{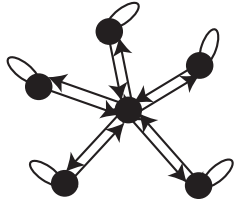}}
\caption{\label{fig:leaf-loop-star} The leaf-loop-star $S_6^{(\ell)}$.  
     \vspace{-5pt}}
\end{center} \vspace{-10pt}
\end{figure}

\begin{rem}\label{obs:kn} 
For $n\geq 4$, the digraph $K_n$ contains at least two composite cycles:  One is an $n$-cycle, and the other is an $(n-2)$-cycle along with a $2$-cycle.  Let $\sP$ be a symmetric nonnegative sign pattern.    If $\Gamma(\sP)=K_n$ with $n\geq 4$, then the signed cycle products of these two composite cycles are   of opposite sign, so $\sP$ is ambiguous.  Similarly, if $\Gamma(\sP)=\Kl_n$ with $n\geq 3$, then $\sP$ is ambiguous.
\end{rem}

The next theorem lists the ten possible initial pairs  $t_1t_2$ that can be attained by $\sepr(\sP)$ for a symmetric nonnegative sign pattern with unique sepr-sequence.  For seven of the ten initial pairs, there is only one sepr-sequence for each order, and this sequence is determined.  After the theorem, Example~\ref{ex:smalldigraphs} is presented to show that each of the three remaining initial pairs has at least two possible sepr-sequences of order four.

\begin{thm}\label{thm:symposunique} Suppose $n\ge 2$ and $t_1\cdots t_n$ is the unique sepr-sequence of a symmetric nonnegative sign pattern (symmetry of the matrices is not assumed).  Then \[t_1t_2\in\{\Ap\Ap,  \NN\Am,\NN\NN, \NN\Sm,  \Sp\Am, \Sp\Ab, \Sp\NN, \Sp\Sp, \Sp\Sm, \Sp\Sb\}.\] 
Furthermore: \ben[$(1)$]
\item\label{A+} If $t_1=\Ap$, then $\Gamma(\sP)=n\Pl_1$ and  $t_1\cdots t_n=\Ap\OL{\Ap}$.
\item\label{NA-} If $t_1t_2=\NN\Am$, then  $n\le 3$ and:
\bit
\item $\Gamma(\sP)=K_2$ and $t_1t_2 =\NN\Am$, or 
\item $\Gamma(\sP)=K_3$ and $t_1t_2 t_3=\NN\Am{\Ap}$.
\eit
\item\label{NN} If $t_1t_2=\NN\NN$, then $\Gamma(\sP)=nP_1$ and $t_1\cdots t_n=\NN\NN\OL{\NN}$.
\item\label{S+A*} If $t_1t_2=\Sp\Ab$, then $\Gamma(\sP)=S_n^{(\ell)}$ and  $t_1\cdots t_n=\Sp\OL{\Ab}\Am$.
\item\label{S+A-} If $t_1t_2=\Sp\Am$, then $n=2$, $\Gamma(\sP)=\Kl_2$, and $t_1t_2 =\Sp\Am$.
\item\label{S+N} If $t_1t_2=\Sp\NN$, then $\Gamma(\sP)=\Pl_1\dunion (n-1)P_1$ and $t_1\cdots t_n=\Sp\NN\OL{\NN}$.
\item\label{S+S+} If $t_1t_2=\Sp\Sp$, then $\Gamma(\sP)=k\Pl_1\dunion (n-k)P_1$ for $2\le k \le n-1$ and $t_1\cdots t_n=\Sp\Sp\OL{\Sp}\OL{\NN}\NN$, where $t_k=\Sp$ and $t_{k+1}=\NN$.
\een  
\end{thm}
\bpf Let $\sP=\begin{bmatrix} p_{ij}\end{bmatrix}$.  Since $\sP$ has a unique sepr-sequence, $\sP$ must have a signed determinant. We consider the three choices $\Ap, \NN,$ and $\Sp$ for $t_1$.  

\smallskip

\noi \textbf{Case 1:} $t_1=\Ap$. 

\eqref{A+}:  Every diagonal entry is positive.  Since  $\sP$ has a signed determinant,   there are no nonzero off-diagonal entries in $\sP$, or equivalently $\Gamma(\sP)$ has no arcs except loops, and every vertex has a loop.

\smallskip

\noi \textbf{Case 2:} $t_1=\NN$.  In this case, every diagonal entry is zero,  so it is not possible to have a positive minor of order two.  Thus $t_2\in\{\Am,\NN,\Sm\}$.

\eqref{NA-}: Since every diagonal entry is zero and every order two principal minor is negative, every off-diagonal entry is positive and  $n\le 3$ by Remark~\ref{obs:kn}.

\eqref{NN}: Every diagonal entry is zero  and every order two sign pattern has zero determinant.  This implies  $\sP$ is the all-zeros sign pattern, or equivalently $\Gamma(\sP)$ has no arcs at all.

\smallskip

\noi \textbf{Case 3:} $t_1=\Sp$.  There is at least one positive  diagonal entry and at least one zero diagonal entry; the zero diagonal entry implies $t_2\ne\Ap$.  Let $\alpha=\{i\in [n]: p_{ii}>0\}$ and $\beta=\{i\in [n]: p_{ii}=0\}$.  Also let $k=|\alpha|$ and $n-k=|\beta|$ with $1<k<n$.

\eqref{S+A*}: Since $t_2(\sP)=\Ab$, every order two principal subpattern has signed determinant equal to $+$ or $-$.  This means $p_{ij}=0$ if $i,j\in\alpha$ for otherwise $\det \sP[\{i,j\}]$ is ambiguous; $p_{ij}=+$ if $i\in\alpha$ and $j\in\beta$ for otherwise $\det \sP[\{i,j\}]=0$; similarly, $p_{ij}=+$ if $i,j\in\beta$ and $i\neq j$ for  otherwise $\det \sP[\{i,j\}]=0$. 

We show that $|\beta|=1$.  If $|\beta|\geq 2$, pick any $v\in\alpha$ and let $\beta'=\beta\cup\{v\}$.  Thus $\Gamma(\sP[\beta'])=\Kl_{|\beta'|}$ with $|\beta'|\geq 3$, and $\sP[\beta']$ is ambiguous by Remark~\ref{obs:kn}.  Consequently, $\sP[\beta']$ along with the positive diagonal entries in $\sP(\beta')$ makes $\sP$ ambiguous, which is a contradiction.  Therefore, $|\beta|=1$ and $\Gamma(\sP)=S_n^{(\ell)}$. 

\eqref{S+A-}: Since every order two principal minor is negative, at most one diagonal entry is positive, and exactly one must be positive since $t_1=\Sp$.   Therefore, every off-diagonal entry is positive and $\Gamma(\sP)=\Kl_n$.  By Remark~\ref{obs:kn} and the assumption that $n\geq 2$, it follows that $n=2$ and $\Gamma(\sP)=\Kl_2$.

\eqref{S+N}: There is at least one positive  diagonal entry and at least one zero diagonal entry.  Since every order two sign pattern has zero determinant,  there are no nonzero off-diagonal entries in $\sP$  and only one loop.

\eqref{S+S+}:  Since there is a positive order 2 principal minor, there are at least two positive  diagonal entries.  Since no order two sign pattern has negative determinant,  there are no nonzero off-diagonal entries in $\sP$.
\epf

\begin{ex}\label{ex:smalldigraphs} 
Let $\sP$ be a symmetric nonnegative sign pattern.  For $i=0,1,2$ the digraph $\Gamma_i$ is shown in Figure \ref{fig:smalldigraphs}.

For $\Gamma(\sP)=P_4$, $\sepr(\sP)=\NN\Sm\NN\Ap$.  
If $\Gamma(\sP)=\Gamma_0$, then $\sepr(\sP)=\NN\Sm\Sp\Ap$.  

If $\Gamma(\sP)=\Gamma_1$, then $\sepr(\sP)=\Sp\Sb\Sm\Ap$.  If $\Gamma(\sP)=\Gamma_2$, then $\sepr(\sP)=\Sp\Sb\Sm\Am$. 

For $\Gamma(\sP)=\Pl_4$, $\sepr(\sP)=\Sp\Sm\Sm\Ap$.  If $\Gamma(\sP)=\Sl_4$, then $\sepr(\sP)=\Sp\Sm\NN\NN$.

\end{ex}

\begin{figure}[h!] \begin{center}
\scalebox{.4}{\includegraphics{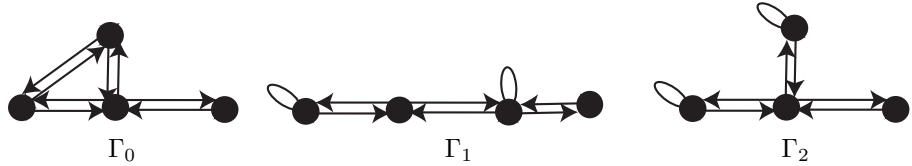}}\\
$\Gamma_0$\hspace{4.1cm}$\Gamma_1$\hspace{4.1cm}$\Gamma_2$
\caption{\label{fig:smalldigraphs} The digraphs for Example \ref{ex:smalldigraphs}  
     \vspace{-5pt}}
\end{center} \vspace{-10pt}
\end{figure}


 \section{Concluding remarks} \label{s:conclude}

We extended the definition of sepr-sequences to sign patterns, and established numerous results about sign patterns that have unique sepr-sequences and about sign semi-stable patterns.  However,  the full characterization of sign patterns with unique sepr-sequences (Conjecture \ref{uniqueconj}) remains open for orders $\geq 5$.  Another type of sign pattern that may be interesting to study is a {\em spectrally arbitrary} sign pattern, that is, one that allows any possible (complex) spectrum attainable by a real matrix.

We made a preliminary study of sepr-sequences of symmetric nonnegative matrices and symmetric nonnegative sign patterns.  Numerous interesting avenues for investigation remain.  For matrices, these include additional forbidden subsequences, families of sepr-sequences that can be realized, and a study of sepr-sequences via adjacency matrices of graphs.  

For seven initial sequences, we determined the one possible sepr-sequence for a symmetric nonnegative sign pattern of order $n$  that has a unique sepr-sequence, but the possible sepr-sequences for three initial starting sequences remain open (in each case there is more than one).  It  would also be interesting to investigate properties of other symmetric sign patterns that have unique sepr-sequences, or more generally find other structural properties that guarantee unique sepr-sequences.  

\section{Acknowledgments}

We thank PIMS for supporting a visit to the University of Victoria by L.H. where this research was initiated.



\end{document}